\documentclass[leqno,12pt]{amsart} 
\usepackage{amssymb} 
\usepackage{braket}
\usepackage{bm}
\usepackage{comment}

\theoremstyle{plain} 
\newtheorem{theorem}{\noindent\bf Theorem}[section] 
\newtheorem{lemma}[theorem]{\noindent\bf Lemma}
\newtheorem{corollary}[theorem]{\noindent\bf Corollary}
\newtheorem{proposition}[theorem]{\noindent\bf Proposition}

\theoremstyle{definition}
\newtheorem{definition}[theorem]{\noindent\bf Definition}

\newtheorem{example}[theorem]{\indent\sc Example}

\newcommand{\Vol}[0]{\operatorname{Vol}}

\newcommand{\deldel}{\sqrt{-1}\partial \overline{\partial}}

\newcommand{\dbar}{\overline{\partial}}
\newcommand{\e}{\varepsilon}

\newcommand{\J}{\mathbb{J}}
\newcommand{\vol}{\mathrm{Vol}}
\begin{document}

\title[Modified extremal metrics and
multiplier Hermitian-Einstein metrics]
{Modified extremal K\"{a}hler metrics and
multiplier Hermitian-Einstein metrics} 
\author[Y. Nakagawa and S. Nakamura]
{Yasuhiro Nakagawa and Satoshi Nakamura} 

\thanks{ 
The first author is supported by JSPS Grant-in-Aid for Scientific
Research (C) No.~23K03094.
The second author is supported by JSPS Grant-in-Aid for Science
Research Activity Start-up No.~21K20342
and Grant-in-Aid for Early-Career Scientists No.~24K16917.
}

\address{
Y. Nakagawa:
Faculty of Advanced Science and Technology, Kumamoto University,
Kurokami 2-40-1, Chuo-ku, Kumamoto 860-8555,
Japan
}
\email{yasunaka@educ.kumamoto-u.ac.jp}

\address{
S. Nakamura:
(Former address) Department of Mathematics, Tokyo Institute of
Technology,
2-12-1, Ookayama, Meguro-ku, Tokyo, 152-8551, Japan}
\address{
S. Nakamura:
(Current address) Department of Mathematics, Institute of Science Tokyo,
2-12-1, Ookayama, Meguro-ku, Tokyo, 152-8551, Japan
}
\email{s.nakamura@math.titech.ac.jp}

\begin{abstract}
Motivated by the notion of multiplier Hermitian-Einstein metric of type
$\sigma$
introduced by Mabuchi,
we introduce the notion of $\sigma$-extremal K\"{a}hler metrics on
compact
K\"{a}hler manifolds,
which generalizes Calabi's extremal K\"{a}hler metrics.
We characterize the existence of this metric in terms of the coercivity
of a certain functional on the space of K\"{a}hler metrics
to show that, on a Fano manifold,
the existence of a multiplier
Hermitian-Einstein metric of type $\sigma$ implies 
 the existence of a $\sigma$-extremal
K\"{a}hler metric.
\end{abstract}

\maketitle

\setcounter{tocdepth}{1}
\tableofcontents

\section{Introduction}
Various generalizations of K\"{a}hler-Einstein metrics are discussed in
the theory for canonical K\"{a}hler metrics.
Calabi's extremal K\"{a}hler metric \cite{Ca82} 
is one of the most important generalizations for a compact K\"{a}hler
manifold.
For a Fano manifold, a Mabuchi soliton \cite{Ma01} 
and a K\"{a}hler-Ricci soliton \cite{Koi90} are also extensively studied 
and can be seen as typical examples of multiplier Hermitian-Einstein
metrics introduced by Mabuchi \cite{Ma03}.

It is known that for a Fano manifold, the existence of a Mabuchi soliton
implies the existence of
Calabi's extremal K\"{a}hler metric in the first Chern class 
as pointed out in \cite[Remark 2.22]{Hi19-2}, \cite[Section 9.6]{Mabook}
and \cite[Remark 5.8]{Ya22}
for example.
In this paper, we focus on this interesting relation for the existence
of different canonical K\"{a}hler metrics from view points of multiplier
Hermitian-Einstein metrics.

Let $M$ be an $n$ dimensional compact connected K\"{a}hler manifold
without boundary and $\Omega$ its K\"{a}hler class.
we denote, by $\mathfrak{X}(M)$, the Lie algebra of holomorphic vector
fields on $M$ and put
$\mathfrak{X}_{0}(M):=\left\{X\in\mathfrak{X}(M)\,|\,
\mathrm{Zero}(X)\ne\varnothing\right\}$,
where $\mathrm{Zero}(X)$ is the zero-set of $X\in\mathfrak{X}(M)$.
Note that if $M$ is a Fano manifold
then $\mathfrak{X}(M)=\mathfrak{X}_{0}(M)$.
For $X\in\mathfrak{X}(M)$ and $\omega\in\Omega$,
there exists a unique complex valued smooth function
$\theta_{X}^{(\omega)}\in C^{\infty}(M;\mathbb{C})$ on $M$ such that
\begin{equation}\label{theta-normalization}
i_{X}\omega=\mathbb{H}_{\omega}\left(i_{X}\omega\right)+
\sqrt{-1}\,\overline{\partial}\theta_{X}^{(\omega)}
\quad\text{and}\quad
\int_{M}\theta_{X}^{(\omega)}\omega^{n}=0,
\end{equation}
where $i_{X}\omega$ is the interior product of $\omega$ with $X$
and $\mathbb{H}_{\omega}(i_{X}\omega)$ the harmonic part of
$i_{X}\omega$ with respect to $\omega$
(cf. \cite[pp.~90--95]{Kobayashi72a}).
Then we have
$$
L_{X}\omega=di_{X}\omega=
\sqrt{-1}\partial\overline{\partial}\theta_{X}^{(\omega)},
$$
where $L_{X}\omega$ is the Lie differentiation of $\omega$ with respect
to $X$. $\theta_{X}^{(\omega)}$ is called the {\it complex potential} of
$X$ with respect to $\omega$.
If $X\in\mathfrak{X}_{0}(M)$, then $\mathbb{H}_{\omega}(i_{X}\omega)=0$,
hence we have
$i_{X}\omega=\sqrt{-1}\,\overline{\partial}\theta_{X}^{(\omega)}$.

Fix $V\in\mathfrak{X}_{0}(M)$ and a $V_{\mathrm{Im}}$-invariant
K\"{a}hler metric $\omega_{0}\in\Omega$,
where $V_{\mathrm{Im}}:=\frac{1}{2\sqrt{-1}}(V-\overline{V})$ is the
imaginary part of $V$. Let 
$$
\mathcal{H}:=\set{ \phi\in{C^{\infty}(M;\mathbb{R})}\,|\,
\omega_{\phi}:=\omega_{0}+\deldel\phi>0}
$$
be the set of all {\it K\"{a}hler potentials}, and put
$\mathcal{H}_{V}=\set{\phi\in\mathcal{H}\,|\,
L_{V_{\mathrm{Im}}}\omega_{\phi}=0}$.
For each $\phi\in\mathcal{H}_{V}$,
$\theta_{V}^{(\omega_{\phi})}\in{C^{\infty}(M;\mathbb{R})}$ is a real
valued function on $M$.
In this article, if there is no fear of confusion,
$\theta_{V}^{(\omega_{\phi})}$ and $\theta_{V}^{(\phi)}$ are used
interchangeably.
According to \cite{FM}, two real numbers
$\min_{M}\theta_{V}^{(\phi)}$ and $\max_{M}\theta_{V}^{(\phi)}$ are
independent of the choice of $\phi\in\mathcal{H}_{V}$.
Set $I=(\alpha,\beta)$
where $$\alpha<\min_{M}\theta_{V}^{(\phi)}
\leqq\max_{M}\theta_{V}^{(\phi)}<\beta.$$
Let $\sigma(s)$ be a real-valued smooth function on $I$.
Since $\int_{M}\exp(-\sigma(\theta_{V}^{(\phi)}))\omega_{\phi}^{n}$ is
independent of the choice of $\phi\in\mathcal{H}_{V}$,
we normalize $\sigma$ by the condition
$$
\int_{M}\exp(-\sigma(\theta_{V}^{(\phi)}))
\omega_{\phi}^{n}=\int_{M}\omega_{\phi}^{n}.
$$

\begin{definition}
A K\"{a}hler metric $\omega_{\phi}$ ($\phi\in\mathcal{H}_{V}$) is
an $\sigma$-{\it extremal K\"{a}hler metric} if
$$S(\omega_{\phi})-\underline{S}=1-e^{-\sigma(\theta_{V}^{(\phi)})}$$
where $S(\omega_{\phi})$ is the scalar curvature of $\omega_{\phi}$ and
$\underline{S}$ its average
$\int_{M}S(\omega_{\phi})\omega_{\phi}^{n}/\int_{M}\omega_{\phi}^{n}$.
We note that the constant $\underline{S}$ is independent of the choice
of $\phi\in\mathcal{H}_{V}$.
\end{definition}

When $1-e^{-\sigma(\theta_{V}^{(\phi)})}=\theta_{V}^{(\phi)}$,
a $\sigma$-extremal K\"{a}hler metric is nothing but an extremal
K\"{a}hler
metric introduced by Calabi \cite{Ca82}.
The first result in this article is to characterize the existence for a
$\sigma$-extremal K\"{a}hler metric
in terms of the properness of a certain energy functional on the space
of K\"{a}hler metrics
as a generalization of the works by Chen-Cheng \cite{CC1, CC2} and He
\cite{He19}.
Let $\mathrm{Aut}_{0}(M)$ be the identity component of the holomorphic
automorphism group for $M$,
and $G=\mathrm{Aut}_{0}(M,V)$ be its subgroup which commutes with the
action of $V$.

\begin{theorem}\label{main thm}
There exists a $\sigma$-extremal K\"{a}hler metric in
$\mathcal{H}_{V}$ if and only if
the $\sigma$-Mabuchi functional is 
coercive (see Definition \ref{coercivity} for the coercivity).
\end{theorem}

We shall prove a more general statement in Theorem~\ref{main thm'} for
an {\it $h$-modified extremal K\"{a}hler metric} in the sense that
$S(\omega_{\phi})-\underline{S}=h(\theta_{V}^{(\phi)})$ for a real
valued smooth function $h$ on the interval $I$ satisfying
$\int_{M}h(\theta_{V}^{(\phi)})\omega_{\phi}^{n}=0$.
Here, we note that the integral
$\int_{M}h(\theta_{V}^{(\phi)})\omega_{\phi}^{n}$ is
independent of the choice of $\phi\in\mathcal{H}_{V}$ \cite[Lemma 2]{Lah19}.

An $h$-modified extremal K\"{a}hler metric is a special case of
a weighted constant scalar curvature K\"{a}hler metric
\cite{AJL23, Lah19, Lah23}.
Indeed, an $h$-modified extremal K\"{a}hler metric 
should be equal to a $(1,h)$-weighted cscK metric.
In our setting, we can extend an argument of Chen-Cheng \cite{CC2} and
He \cite{He19}
on the method of continuity to show the regularity of minimizers for
a Mabuchi type energy functional corresponds to $h$-modified extremal
K\"{a}hler metrics
in Section~\ref{Reg of mini}.
As the result, we can establish not only the coercivity theorem
(Theorem~\ref{main thm'})
but also the geodesic stability
theorems (Theorems~\ref{uniform stable thm} and \ref{geod stability}).

The definition of the $\sigma$-extremal K\"{a}hler metric is motivated by
both the notion of the multiplier Hermitian-Einstein metrics of type
$\sigma$
introduced by Mabuchi \cite{Ma03} and the inequality \eqref{KgeqD}.
Consider the case when $M$ is a Fano manifold and $\Omega=2\pi
c_{1}(M)$.
Assume one of the following conditions:
(i) $\dot{\sigma}\leqq{0}\leqq\ddot{\sigma}$, 
(ii) $\ddot{\sigma} > 0$,
where $\dot{\sigma}$ and $\ddot{\sigma}$ are the first derivative and
the second derivative respectively.
Fix $\omega_{\phi}\in\Omega$.
Mabuchi called a conformally K\"{a}hler metric
$$
\widetilde{\omega}_{\phi}:=
\omega_{\phi}\exp\left(-\frac{1}{n}\sigma(\theta_{V}^{(\phi)})\right)
$$ 
a {\it multiplier Hermitian metric of type $\sigma$}.
Note that the multiplier Hermitian metric $\widetilde{\omega}_{\phi}$
can be seen as an Hermitian metric on the holomorphic tangent bundle
$TM$. Then $\widetilde{\omega}_{\phi}$ defines the Hermitian connection
$$
\widetilde{\nabla}:=\nabla -\frac{\partial
(\sigma(\theta_{V}^{(\phi)}))}{n} \mathrm{id}_{TM},
$$
where $\nabla$ is the natural connection with respect to $\omega$.
The Ricci form $\mathrm{Ric}^{\sigma}_{V}(\omega_{\phi})$ of
$(\widetilde{\omega}_{\phi}, \widetilde{\nabla})$ is equal to
$\mathrm{Ric}(\omega_{\phi})+\deldel\sigma(\theta_{V}^{(\phi)})$,
where $\mathrm{Ric}(\omega_{\phi})\in{2\pi}c_{1}(M)$ is the Ricci form
for $\omega_{\phi}$ defined by $-\deldel\log\omega_{\phi}^{n}$.

\begin{definition}
The conformally K\"{a}hler metric $\widetilde{\omega}_{\phi}$ is
a {\it multiplier Hermitian-Einstein metric of type $\sigma$}
if $\mathrm{Ric}^{\sigma}_{V}(\omega_{\phi})=\omega_{\phi}$.
Then we call $\omega_{\phi}$ itself a {\it $\sigma$-soliton}.
\end{definition}

Define the Ricci potential $\rho_{\phi}$ for $\omega_{\phi}$ as
follows.
\begin{equation}\label{Ricci}
 \mathrm{Ric}(\omega_{\phi})-\omega_{\phi}=\deldel\rho_{\phi}
 \quad\text{and}\quad
 \int_{X}(1-e^{\rho_{\phi}})\omega_{\phi}^{n}=0.
\end{equation}
In this terminology, $\omega_{\phi}$ is a $\sigma$-soliton
if and only if
$\rho_{\phi}+\sigma(\theta_{V}^{(\phi)})=0$.

A $\sigma$-soliton gives
some well-known generalizations of K\"{a}hler-Einstein metrics. 
\begin{enumerate}
\item[(i)] When $V=0$,
a $\sigma$-soliton $\omega_{\phi}$ is nothing but
a K\"{a}hler-Einstein metric.
\item[(ii)] When $\sigma(s) =-s+C$, where $C$ is a constant,
$\sigma$-soliton $\omega_{\phi}$ is a {\it K\"{a}hler-Ricci soliton}
in the sense that
$\mathrm{Ric}(\omega_{\phi})-\omega_{\phi}=L_{V}\omega_{\phi}$.
\item[(iii)] When $\sigma(s)=-\log(s+1)$ and $\theta_{V}^{(\phi)}>-1$,
$\sigma$-soliton $\omega_{\phi}$ is
a {\it Mabuchi soliton} (\cite{Hi19-2, Ma01, Ya21})
in the sense that $1-e^{\rho_{\phi}}$ is the potential function of the
holomorphic vector field $-V$ with respect to $\omega_{\phi}$. 
\end{enumerate}

The existence of a $\sigma$-soliton is well understood.
Han-Li \cite{HaLi20} showed that the existence is equivalent to an
algebraic stability condition for a Fano manifold.
The authors \cite{NN22} established a combinatorial condition to
characterize the existence for a $\sigma$-soliton
on a KSM-manifold which is a certain toric bundle
including a toric Fano manifold.

As an application of Theorem~\ref{main thm},
we obtain a relation of the existence for
a $\sigma$-soliton and a $\sigma$-extremal K\"{a}hler metric,
which is the second result in this article.

\begin{theorem}\label{2nd thm}
Let $M$ be a Fano manifold admitting a $\sigma$-soliton in
$\Omega=2\pi{c_{1}(M)}$.
Then $M$ also admits a $\sigma$-extremal K\"{a}hler metric in
$\Omega=2\pi{c_{1}(M)}$.
\end{theorem}

When $\sigma(s)=-\log(s+1)$, this relation is already known for the
literature
as stated in the beginning of the paper.
We note that a $\sigma$-soliton and a $\sigma$-extremal metric are
different in general. See Section~\ref{sigma-soliton}.

\subsection*{Organization}
In Section~\ref{Preliminary} various energy functionals are introduced.
In particular the $h$-modified Mabuchi functional 
 whose
critical points are
$h$-modified extremal K\"{a}hler metrics is introduced.
In Section~\ref{exist/proper},
based on the existence\,/\,properness principle established by
Darvas-Rubinstein
\cite{Dar19,DR17}, we prove Theorem~\ref{main thm}.
In Section~\ref{gstability}, based on the work of Chen-Cheng \cite{CC2},
we show that the the existence of an $h$-modified 
extremal K\"{a}hler metric is equivalent to some notions of geodesic
stability.
In Section~\ref{sigma-soliton}, we prove Theorem~\ref{2nd thm} and
provide
an example of an $h$-modified extremal K\"{a}hler metric related to
a K\"{a}hler-Ricci soliton.

\subsection*{Acknowledgement}
The authors thank the referees for their careful reading and
for numerous useful suggestions which improved the presentation of the paper.
\section{Preliminaries}\label{Preliminary}
\subsection{The metric completion of $\mathcal{H}_{V}$}
Following the works of Darvas \cite{Dar15, Dar17}, 
we introduce the completion of $\mathcal{H}$.
For $p \geqq{1}$, the length of a smooth curve
$[0,1]\ni t \mapsto c_{t}\in\mathcal{H}$ is given by
$\int_{0}^{1}\| \dot{c_{t}} \|_{L^{p}(M, \omega_{c_{t}})}dt.$
For any $\phi_{0},\phi_{1}\in\mathcal{H}$, the distance
$d_{p}(\phi_{0},\phi_{1})$ 
is defined as the infimum of the length of smooth curves connecting
$\phi_{0}$ and $\phi_{1}$.
The pair $(\mathcal{H}, d_{p})$ is in fact a metric space \cite{Dar15}.
Let
$$
\mathrm{PSH}(M, \omega_{0}):=
\Set{ \phi\in L^{1}(M, \omega_{0}) | \phi
\text{ is u.s.c. and }\omega_{\phi}\geqq{0}}
$$
and
$$
\mathcal{E}^{p}(M, \omega_{0}):=
\Set{ \phi\in\mathrm{PSH}(M, \omega_{0}) |
\int_{M}\omega_{\phi}^{n}=\Vol, \quad
\int_{M}|\phi|^{p}\omega_{\phi}^{n}<+\infty}.
$$
Darvas \cite{Dar15} showed that
the completion of the metric space $(\mathcal{H}, d_{p})$ is identified
with
$(\mathcal{E}^{p}, d_{p})$, where
$d_{p}(u_{0}, u_{1})=
 \lim_{k\to\infty}d_{p}\bigl(u_{0}(k),u_{1}(k)\bigr)$
for any smooth decreasing sequences $\{ u_{i}(k) \}_{k\in\mathbb{N}}$ in
$\mathcal{H}$
converging pointwise to $u_{i}\in\mathcal{E}^{p}$.
We will need the case when $p=1$ for our purpose.
According to \cite[Lemma~5.4]{DR17}, we can take the completion of
$(\mathcal{H}_{V}, d_{1})$
which we denote by $(\mathcal{E}_{V}^{1}, d_{1})$.

\subsection{Energy functionals}
At first, we work on a compact K\"{a}hler manifold $M$.
Set $\mathrm{Vol}:=\int_{M}\omega_{0}^{n}$.
For any $\phi\in\mathcal{H}$, we define the
{\it Aubin-Yau} (or {\it Aubin-Mabuchi}) {\it functional} as follows.
$$
E(\phi)=E_{\omega_{0}}(\phi):=
\frac{\mathrm{Vol}^{-1}}{n+1}\int_{M}\phi\sum_{j=0}^{n}\omega_{0}^{j}
\wedge\omega_{\phi}^{n-j}.
$$
We also define Aubin's I-{\it functional} and J-{\it functional} as
follows.
\begin{align*}
&I(\phi)=I_{\omega_{0}}(\phi):=
\vol^{-1}\int_{M}\phi(\omega_{0}^{n}-\omega_{\phi}^{n}),\\
&J(\phi)=J_{\omega_{0}}(\phi):=
\vol^{-1}\int_{M}\phi\omega_{0}^{n}-E_{\omega_{0}}(\phi).
\end{align*}
Set $\J(\phi)=I(\phi)-J(\phi)$.
Note that they satisfy the classical inequality
\begin{equation}\label{IJ}
0\leqq{I(\phi)}\leqq{(n+1)\J(\phi)}\leqq{nI(\phi)}.
\end{equation}
For $(1,1)$-form $\chi$, we define
$$
\J_{\chi}(\phi)=\J_{\chi, \omega_{0}}(\phi):=
\mathrm{Vol}^{-1}\int_{M}\phi\sum_{j=0}^{n-1}
\chi\wedge\omega_{0}^{j}\wedge\omega_{\phi}^{n-1-j}
-\underline{\chi}E(\phi)
$$
where
$\underline{\chi}=n\int_{M}\chi\wedge\omega^{n-1}/\int_{M}\omega^{n}$.
In particular, we denote
$$
\J_{-\mathrm{Ric}}(\phi):=\underline{S}E(\phi)
-\mathrm{Vol}^{-1}\int_{M}\phi\sum_{j=0}^{n-1}
\mathrm{Ric}(\omega_{0})\wedge\omega_{0}^{j}\wedge\omega_{\phi}^{n-1-j}.
$$
The entropy functional is defined as
$$
H(\phi)=H_{\omega_{0}}(\phi):=
\mathrm{Vol}^{-1}\int_{M}\log
\Bigl(\frac{\omega_{\phi}^{n}}{\omega_{0}^{n}}\Bigr) 
\omega_{\phi}^{n}.
$$
Fix a real valued function $h$ on the interval $I$
satisfying $\int_{M}h(\theta_{V}^{(\phi)})\omega_{\phi}^{n}=0$.
For any $\phi\in\mathcal{H}_{V}$, we introduce
the {\it $h$-modified Aubin-Yau functional}
$$
\J_{V}^{h}(\phi):=
\mathrm{Vol}^{-1}\int_{0}^{1}dt\int_{M}\frac{d\phi_{t}}{dt}
h(\theta_{V}^{(\phi_{t})})\omega_{\phi_{t}}^{n},
$$
where $\phi_{t}\in\mathcal{H}_{V}$ is a smooth path connecting $0$
and $\phi$.
According to \cite[Lemma 2.14]{BW14} and \cite[Lemma 3]{Lah19},
the functional $\mathbb{J}_{V}^{h}$ is independent of the choice of the
path.
We define
$$
E_{V}^{h}:=E-\J_{V}^{h}
$$
and to define the {\it $h$-modified Mabuchi functional} as
$$
K_{V}^{h}:=H+\J_{-\mathrm{Ric}}+\J_{V}^{h}.
$$
An {\it $h$-modified extremal K\"{a}hler metric} which is defined by the
equation
$S(\omega_{\phi})-\underline{S}=h(\theta_{V}^{(\phi)})$
is a critical point of $K_{V}^{h}$ since
\begin{equation}\label{dK}
\delta K_{V}^{h}=-\mathrm{Vol}^{-1}\int_{M}\delta\phi
\Bigl( S(\phi)-\underline{S}-h(\theta_{V}^{(\phi)})
\Bigr)\omega_{\phi}^{n}.
\end{equation}
When $h(s)=1-e^{-\sigma(s)}$, we call the functional $K_{V}^{h}$ the
$\sigma$-{\it{Mabuchi functional}}.

\begin{lemma}\label{Jaffine}
The following hold.
\begin{enumerate}
\item[(i)] When $1-h\geqq{0}$ on $I$, for any $\phi_{0}\leqq\phi_{1}$,
we have $E_{V}^{h}(\phi_{0})\leqq{E_{V}^{h}(\phi_{1})}$.
\item[(ii)] {\rm (\cite[Proposition 2.17]{BW14})}
The functions $E_{V}^{h}$ and $\J_{V}^{h}$ are affine along
any $C^{1,1}$-geodesic connecting two points in $\mathcal{H}_{V}$.
\end{enumerate}
\end{lemma}

\begin{proof}
First we prove (i).
Note that for any smooth path $\phi_{t}\in\mathcal{H}_{V}$ connecting
$\phi_{0}$ and $\phi_{1}$,
$$
E_{V}^{h}(\phi_{1})-E_{V}^{h}(\phi_{0})
=\mathrm{Vol}^{-1}\int_{0}^{1}dt\int_{M}\dot{\phi_{t}}
(1-h(\theta_{V}^{(\phi_{t})}))\omega_{\phi_{t}}^{n}.
$$
Thus taking $\phi_{t}:=t(\phi_{1}-\phi_{0})+\phi_{0}$, we have
$E_{V}^{h}(\phi_{0})\leqq{E_{V}^{h}(\phi_{1})}$.

Next we prove (ii). 
The affineness for $E$ along any $C^{1,1}$-geodesic segment is well
known. 
We prove $E_{V}^{h}$ is affine.
The following proof is same as \cite[Proposition~10.4]{Bern15} and
\cite[Lemma~2.5]{LZ}.
Take a unique $C^{1,1}$-geodesic $\phi_{t}$ connecting
$\phi_{0},\,\phi_{1}\in\mathcal{H}_{V}$
and small $\e>0$.
By \cite{C20} the geodesic $\phi_{t}$ can be approximated by smooth
$\e$-geodesics
$\phi_{t}^{\e}\in\mathcal{H}_{V}$ connecting $\phi_{0}, \phi_{1}$
defining
$$
\Big(
\frac{\partial^{2}\phi_{t}^{\e}}{\partial\tau\partial\bar{\tau}}
-\Big|
\bar{\partial}\Big( \frac{\partial\phi_{t}^{\e}}{\partial\tau}
\Big)
\Big|^{2}_{\omega_{\phi_{t}^{\e}}}
\Big)
\omega_{\phi_{t}^{\e}}^{n}
=\e\omega_{0}^{n}
$$
on
$M\times([0,1]\times S^{1})\subset M\times\mathbb{C}$
with $t=\mathrm{Re}(\tau)$.
The argument by integration by parts in the proof of
\cite[Lemma~2.5]{LZ} 
(we need to replace
$1-\theta_{X}(\omega_{\phi_{t}^{\e}})$ by 
$1-h(\theta_{V}^{\phi_{t}^{\e}})$)
shows
\begin{eqnarray*}
\frac{\partial^{2}}{\partial\tau\partial\bar{\tau}}
E_{V}^{h}(\phi_{t}^{\e})
&=&
-\mathrm{Vol}^{-1}\int_{M}
\Big(
\frac{\partial^{2}\phi_{t}^{\e}}{\partial\tau\partial\bar{\tau}} 
-\Big|
\bar{\partial}\Big( \frac{\partial\phi_{t}^{\e}}{\partial\tau}
\Big)
\Big|^{2}_{\omega_{\phi_{t}^{\e}}}
\Big)
(1-h(\theta_{V}^{\phi_{t}^{\e}}))\omega_{\phi_{t}^{\e}}^{n}\\
&=&
-\e\mathrm{Vol}^{-1}\int_{M}
(1-h(\theta_{V}^{\phi_{t}^{\e}}))\omega_{0}^{n}.
\end{eqnarray*}
Since there exists a uniform constant $C$ 
independent of $\phi\in\mathcal{H}_{V}$
such that $|\theta_{V}^{(\phi)}|<C$,
it follows that
$$
\Big|\frac{\partial^{2}}{\partial\tau\partial\bar{\tau}}
E_{V}^{h}(\phi_{t}^{\e})\Big|<\e C.
$$
When $\e\to{0}$,
$\deldel_{\tau}E_{V}^{h}(\phi_{t}^{\e})$ converges weakly to
$\deldel_{\tau}E_{V}^{h}(\phi_{t})$ as Monge-Amp\`{e}re measures,
where $\partial\overline{\partial}_{\tau}$ is
$\partial\overline{\partial}$-operator with respect to
$\tau\in\mathbb{C}$.
Thus $\deldel_{\tau}E_{V}^{h}(\phi_{t})=0$.
This completes the proof.
\end{proof}

The above lemma together with the convexity of the K-energy
$K:=H+\J_{-\mathrm{Ric}}$
proved by Berman-Berndtsson \cite{BB17} shows that the functional
$K_{V}^{h}$
is convex along $C^{1,1}$ geodesic segments in $\mathcal{H}_{V}$,
and a $h$-modified extremal metric minimizes $K_{V}^{h}$.

Berman-Darvas-Lu \cite{BDL20} showed that
the K-energy $K:=H+\mathbb{J}_{-\mathrm{Ric}}$ on $\mathcal{H}$ can be
extended to a functional
$\mathcal{E}^{1}\to\mathbb{R}\cup\{+\infty\}$ using the same formula
and such a K-energy in $\mathcal{E}^{1}$ is the greatest $d_{1}$-lsc
extension of $K$ on $\mathcal{H}$.
Moreover they showed the K-energy is convex along finite energy
geodesics in $\mathcal{E}^{1}$.
These fact and the following lemma show that
the functional $K_{V}^{h}$ can be extended to a functional
$\mathcal{E}^{1}_{V}\to\mathbb{R}\cup\{+\infty\}$
and the extended functional $K_{V}^{h}$ on $\mathcal{E}^{1}_{V}$ is the
greatest $d_{1}$-lsc extension of
$K_{V}^{h}$ on $\mathcal{H}_{V}$.
Moreover, $K_{V}^{h}$ is convex along finite energy geodesics in
$\mathcal{E}^{1}_{V}$,
and a $h$-modified extremal K\"{a}hler  metric is a minimizer on
$\mathcal{E}^{1}_{V}$.

\begin{lemma}
{\rm (\cite[Lemma 6.10]{AJL23})}
The function $\J_{V}^{h}$ on $\mathcal{H}_{V}$ has a unique
$d_{1}$-continuous
extension to $\mathcal{E}^{1}_{V}$.
\end{lemma}

\begin{proof}
Note that there exists a uniform constant $C$ 
independent of $\phi\in\mathcal{H}_{V}$ by \cite{Zh00}
such that $|\theta_{V}^{(\phi)}|<C$.
Take any smooth curve $\phi_{t}$ in $\mathcal{H}_{V}$ connecting
$\phi_{0}$ and $\phi_{1}$.
Observe
$$
|\J_{V}^{h}(\phi_{0})-\J_{V}^{h}(\phi_{1})|
\leqq{C\int_{0}^{1}dt\int_{M}|\dot{\phi_{t}}|\omega_{\phi_{t}}^{n}}
$$
to show
$|\J_{V}^{h}(\phi_{0})-\J_{V}^{h}(\phi_{1})|
\leqq{C}d_{1}(\phi_{0},\phi_{1})$
by taking infimum over all smooth curves.
\end{proof}

Consider the case when $M$ is Fano and $\Omega=2\pi c_{1}(M)$.
We define
$$
L(\phi):=-\log
\Bigl(\mathrm{Vol}^{-1}\int_{M}e^{\rho_{0}-\phi}\omega_{0}^{n}\Bigr)
$$
to introduce the {\it $h$-modified Ding functional} as
$$
D_{V}^{h}:=-E_{V}^{h}+L.
$$
A $\sigma$-soliton is a critical point of $D_{V}^{h}$ with
$h(s)=1-e^{-\sigma(s)}$ since
$$
\delta D_{V}^{h}=-\mathrm{Vol}^{-1}\int_{M}\delta\phi
\Bigl(1-h(\theta_{V}^{\phi}) -e^{\rho_{\phi}} \Bigr) \omega_{\phi}^{n}.
$$
We call the functional $D_{V}^{h}$ with $h(s)=1-e^{-\sigma(s)}$
the {\it $\sigma$-Ding functional}. 

The inequality in the following lemma motivated to introduce the notion
of an $\sigma$-extremal K\"{a}hler metric.
This is well-known for the literature in the case when $V=0$ and in the
case when $h(s)=s$.
See \cite[Section 9.6]{Mabook} for example.

\begin{lemma}
Let $M$ be a Fano manifold with $\Omega=2\pi c_{1}(M)$.
For any $\phi\in\mathcal{H}_{V}$, we have
\begin{equation}\label{FanoK}
K_{V}^{h}(\phi)
=\mathrm{Vol}^{-1}\int_{M}\log 
\Bigl(\frac{\omega_{\phi}^{n}}{e^{\rho_{0}}\omega_{0}^{n}}\Bigl)
\omega_{\phi}^{n}
-E_{V}^{h}(\phi)+\mathrm{Vol}^{-1}\int_{M}\phi\omega_{\phi}^{n}
+\mathrm{Vol}^{-1}\int_{M}\rho_{0}\omega_{0}^{n}.
\end{equation}
As a consequence, for any $\phi\in\mathcal{H}_{V}$ we have
\begin{equation}\label{KgeqD}
K_{V}^{h}(\phi)-\mathrm{Vol}^{-1}\int_{M}\rho_{0}\omega_{0}^{n}
\geqq{D_{V}^{h}(\phi)}.
\end{equation}
\end{lemma}

\begin{proof}
It is easy to see that the derivative of the RHS in \eqref{FanoK} equals
to the RHS in \eqref{dK}.
We note that the value of the RHS in \eqref{FanoK} is $0$ at $\phi=0$ to
show the equality.
Next we prove the inequality~\eqref{KgeqD}. We can assume
$\int_{M}e^{\rho_{0}-\phi}\omega_{0}^{n}=\Vol$.
This is because $K_{V}^{h}$ and $D_{V}^{h}$ are invariant under the
translation $\phi\mapsto\phi+\text{constant}$.
Observe that Jensen's inequality for the convex function $x\log x$
yields 
$$
\mathrm{Vol}^{-1}\int_{M}\log
\Bigl( \frac{\omega_{\phi}^{n}}{e^{\rho_{0}-\phi}\omega_{0}^{n}}\Bigl)
\omega_{\phi}^{n}
=\Vol^{-1}\int_{M}
\frac{\omega_{\phi}^{n}}{e^{\rho_{0}-\phi}\omega_{0}^{n}}\log
\Bigl( \frac{\omega_{\phi}^{n}}{e^{\rho_{0}-\phi}\omega_{0}^{n}}\Bigl)
e^{\rho_{0}-\phi}\omega_{0}^{n}
\geqq{0}.
$$
Therefore
$$
K_{V}^{h}(\phi)-\Vol^{-1}\int_{M}\rho_{0}\omega_{0}^{n}
\geqq{-E_{V}^{h}(\phi)}
=D_{V}^{h}(\phi).
$$
\end{proof}

\subsection{The notion of the coercivity}
Set $G:=\mathrm{Aut}_{0}(M,V)$,
$\mathcal{H}_{0}:=\mathcal{H}\cap E^{-1}(0)$ and
$\mathcal{H}_{V,0}:=\mathcal{H}_{V}\cap E^{-1}(0)$.
For any K\"{a}hler metric $\omega_{\phi}=\omega_{0}+\deldel\phi$
and $\sigma\in\mathrm{Aut}_{0}(M)$, we define a K\"{a}hler potential 
$\sigma[\phi]\in\mathcal{H}_{0}$ by
$\sigma^{*}\omega_{\phi}=\omega_{0}+\deldel\sigma[\phi]$.
Then for any $\phi, \psi \in\mathcal{H}_{0}$ the $d_{1}$-distance
relative to $G$ is defined by
$d_{1,G}(\phi,\psi):=
\inf_{\sigma_{1},\sigma_{2}\in{G}}
d_{1}(\sigma_{1}[\phi],\sigma_{2}[\psi])$
which is also equal to $\inf_{\sigma\in G}d_{1}(\phi,\sigma[\psi])$.
This distance can be defined on $\mathcal{H}_{V,0}$,
since the $G$-action preserves $\mathcal{H}_{V,0}$ 
and the inclusion $\mathcal{H}_{V,0}\subset\mathcal{H}_{0}$ is totally
geodesic.
For later use, we also define
$\mathcal{E}^{1}_{V,0}:=\mathcal{E}^{1}_{V}\cap E^{-1}(0)$
which is the completion of $\mathcal{H}_{V,0}$.

\begin{definition}\label{coercivity}
A 
functional $F$ on $\mathcal{H}_{V,0}$ is {\it coercive}
if
there exist constants $C,D>0$ such that for all
$\phi\in\mathcal{H}_{V,0}$,
$$
F(\phi)\geqq{Cd_{1,G}(0,\phi)-D}.
$$
\end{definition}

\subsection{Futaki type invariants}
Differentiating the functional $K_{V}^{h}$, we get a Futaki type
invariant for $M$.
For
$X\in\mathfrak{X}_{0}(M)$
and $\phi\in\mathcal{H}_{V}$, we put
$$
F_{V}^{h}(X):=-\int_{M}
\theta_{X}^{(\omega_{\phi})}
\Big(S(\omega_{\phi})-\underline{S}
-h(\theta_{V}^{(\omega_{\phi})})\Big)\omega_{\phi}^{n}.
$$
Since
$\int_{M}\left(S(\omega_{\phi})-\underline{S}\right)\omega_{\phi}^{n}=0$
and $\int_{M}h(\theta_{V}^{(\omega_{\phi})})\omega_{\phi}^{n}=0$,
there exist
$\rho_{\omega_{\phi}},\widetilde{h}_{V}^{(\omega_{\phi})}
\in{C^{\infty}(M;\mathbb{R})}$
satisfying
\begin{align*}
&S(\omega_{\phi})-\underline{S}=
\square_{\omega_{\phi}}\rho_{\omega_{\phi}}
:=\sum_{i,j=1}^{n}g^{\overline{\jmath}i}
\frac{\partial^{2}\rho_{\omega_{\phi}}}
{\partial{z^{i}}\partial\overline{z^{\jmath}}},\\
&\int_{M}e^{\rho_{\omega_{\phi}}}\omega_{\phi}^{n}=
\int_{M}\omega_{\phi}^{n},\\
&h(\theta_{V}^{(\omega_{\phi})})=
\square_{\omega_{\phi}}\widetilde{h}_{V}^{(\omega_{\phi})},
\end{align*}
where
$\omega_{\phi}=\sqrt{-1}\sum_{i,j=1}^{n}
g_{i\overline{\jmath}}dz^{i}\wedge\overline{dz^{\jmath}}$
and
$\left(g^{\overline{\jmath}i}\right)_{i,j=1}^{n}$ is the inverse matrix
of $\left(g_{i\overline{\jmath}}\right)_{i,j=1}^{n}$.
Then we have
$$
F_{V}^{h}(X)=
\int_{M}
X\left(\rho_{\omega_{\phi}}-\widetilde{h}_{V}^{(\omega_{\phi})}\right)
\omega_{\phi}^{n}.
$$

\begin{proposition}\label{prop:futaki-inv}
The value $F_{V}^{h}(X)$ is independent of the choice of
$\phi\in\mathcal{H}_{V}$.
In particular, if there exists an $h$-modified extremal K\"{a}hler
metric
in $\mathcal{H}_{V}$, the invariant $F_{V}^{h}$ vanishes identically.
\end{proposition}

\begin{proof}
Since
$-\int_{M}\theta_{X}^{(\omega_{\phi})}\left(S(\omega_{\phi})
-\underline{S}\right)\omega_{\phi}^{n}
=\int_{M}X\left(\rho_{\omega_{\phi}}\right)\omega_{\phi}^{n}$
is the original Futaki invariant (\cite{Fu83}, see also \cite{Fubook}),
this is independent of the choice of $\phi\in\mathcal{H}_{V}$.

For an arbitrary smooth path
$\left\{\phi_{t}\,|\,a\leqq{t}\leqq{b}\right\}$ in
$\mathcal{H}_{V}$,
by a simple calculation, we obtain
$$
\frac{d}{dt}\theta_{X}^{(\omega_{\phi_{t}})}
=X\!\left(\dot{\phi}_{t}\right),\qquad
\frac{d}{dt}h(\theta_{V}^{(\omega_{\phi_{t}})})
=V\!\left(\dot{\phi}_{t}\right)
h^{\prime}(\theta_{V}^{(\omega_{\phi_{t}})}),\qquad
\frac{d}{dt}\omega_{\phi_{t}}^{n}
=\left(\square_{\omega_{\phi_{t}}}\dot{\phi}_{t}\right)
\omega_{\phi_{t}}^{n}.
$$
In view of these equalities, we have
\begin{align*}
&\frac{d}{dt}\int_{M}\theta_{X}^{(\omega_{\phi_{t}})}
h(\theta_{V}^{(\omega_{\phi_{t}})})\omega_{\phi_{t}}^{n}\\
&=
\int_{M}\left\{
X\!\left(\dot{\phi}_{t}\right)h(\theta_{V}^{(\omega_{\phi_{t}})})
+\theta_{X}^{(\omega_{\phi_{t}})}
h^{\prime}(\theta_{V}^{(\omega_{\phi_{t}})})V\!\left(\dot{\phi}_{t}\right)
+\theta_{X}^{(\omega_{\phi_{t}})}h(\theta_{V}^{(\omega_{\phi_{t}})})
\square_{\omega_{\phi_{t}}}\dot{\phi}_{t}
\right\}\omega_{\phi_{t}}^{n}\\
&=
\int_{M}\left\{
X\!\left(\dot{\phi}_{t}\right)h(\theta_{V}^{(\omega_{\phi_{t}})})
+\theta_{X}^{(\omega_{\phi_{t}})}
h^{\prime}(\theta_{V}^{(\omega_{\phi_{t}})})V\!\left(\dot{\phi}_{t}\right)
\right\}\omega_{\phi_t}^{n}\\
&\qquad-\int_{M}
g_{\phi_{t}}\left(
\overline{\partial}\left(
\theta_{X}^{(\omega_{\phi_{t}})}h(\theta_{V}^{(\omega_{\phi_{t}})})\right),
\,\overline{\partial}\overline{\dot{\phi}_{t}}\right)
\omega_{\phi_t}^{n}\\
&=0,
\end{align*}
where $g_{\phi_{t}}$ is the fiberwise Hermitian inner product for
$(0,1)$-forms induced by $\omega_{\phi_{t}}$.
Hence,
$\int_{M}\theta_{X}^{(\omega_{\phi})}
h(\theta_{V}^{(\omega_{\phi})})\omega_{\phi}^{n}$ is
independent of the choice of $\phi\in\mathcal{H}_{V}$.
This completes the proof.
\end{proof}

We note that we can take $V=0$ provided the original Futaki invariant
vanishes.

We consider the case when $M$ is Fano and $\Omega=2\pi{c_{1}(M)}$.
For $\omega\in\Omega=2\pi{c_{1}(M)}$ and
$X\in\mathfrak{X}_{0}(M)=\mathfrak{X}(M)$, we define
$\widetilde{\theta}_{V}^{(\omega)}\in{C^{\infty}(M;\mathbb{C})}$ by
$$
i_{X}\omega=
\sqrt{-1}\,\overline{\partial}\widetilde{\theta}_{X}^{(\omega)}
\quad\text{and}\quad
\int_{M}\widetilde{\theta}_{X}^{(\omega)}e^{\rho_{\omega}}\omega^{n}=0,
$$
where $\rho_{\omega}$ is the Ricci potential for $\omega$.
Then we have
$\square_{\omega}\widetilde{\theta}_{X}^{(\omega)}
+X\left(\rho_{\omega}\right)
+\widetilde{\theta}_{X}^{(\omega)}=0$
(cf. Tian-Zhu \cite[(2.1)]{TZ02}), hence we obtain
$$
\int_{M}\widetilde{\theta}_{X}^{(\omega)}\omega^{n}
=-\int_{M}X\left(\rho_{\omega}\right)\omega^{n}
$$
and
$\widetilde{\theta}_{X}^{(\omega)}=\theta_{V}^{(\omega)}
-\mathrm{Vol}^{-1}\int_{M}X\left(\rho_{\omega}\right)\omega^{n}$.
In view of these equalities, we conclude that
{\allowdisplaybreaks
\begin{align*}
F_{V}^{h}(X)
&=\int_{M}X\left(\rho_{\omega}\right)\omega^{n}+
\int_{M}\theta_{X}^{(\omega)}h(\theta_{V}^{(\omega)})\omega^{n}\\
&=-\int_{M}\widetilde{\theta}_{X}^{(\omega)}\omega^{n}+
\int_{M}\widetilde{\theta}_{X}^{(\omega)}
h(\theta_{V}^{(\omega)})\omega^{n}\\
&=-\int_{M}\widetilde{\theta}_{X}^{(\omega)}
\left(1-e^{\rho_{\omega}}\right)\omega^{n}+
\int_{M}\widetilde{\theta}_{X}^{(\omega)}
h(\theta_{V}^{(\omega)})\omega^{n}\\
&=-\int_{M}\widetilde{\theta}_{X}^{(\omega)}
\left(1-e^{\rho_{\omega}}
-h(\theta_{V}^{(\omega)})\right)\omega^{n}\\
&=-\int_{M}{\theta}_{X}^{(\omega)}
\left(1-e^{\rho_{\omega}}
-h(\theta_{V}^{(\omega)})\right)\omega^{n}.
\end{align*}}\noindent
Therefore, we can see that this invariant $F_{V}^{h}(X)$ is the
derivative of the $h$-modified Ding functional $D_{V}^{h}$
along the ray generated by $X$.
Moreover, assuming $h(s)=1-e^{-\sigma(s)}$,
since
$\int_{M}\theta_{X}^{(\omega)}e^{\rho_{\omega}}\omega^{n}=
\int_{M}X\left(\rho_{\omega}\right)\omega^{n}$,
we can conclude
{\allowdisplaybreaks
\begin{align*}
F_{V}^{h}(X)
&=-\int_{M}\theta_{X}^{(\omega)}
\left(e^{-\sigma(\theta_{V}^{(\omega)})}
-e^{\rho_{\omega}}\right)\omega^{n}\\
&=\int_{M}\theta_{X}^{(\omega)}e^{\rho_{\omega}}\omega^{n}
-\int_{M}\theta_{X}^{(\omega)}
e^{-\sigma(\theta_{V}^{(\omega)})}\omega^{n}\\
&=\int_{M}X\left(\rho_{\omega}\right)\omega^{n}
-\int_{M}\theta_{X}^{(\omega)}
e^{-\sigma(\theta_{V}^{(\omega)})}\omega^{n}\\
&=-\int_{M}\widetilde{\theta}_{X}^{(\omega)}
e^{-\sigma(\theta_{V}^{(\omega)})}\omega^{n}\\
&=\int_{M}\left(
\square_{\omega}\widetilde{\theta}_{X}^{(\omega)}
+X\left(\rho_{\omega}\right)\right)
e^{-\sigma(\theta_{V}^{(\omega)})}\omega^{n}\\
&=\int_{M}X
\left(\rho_{\omega}+\sigma(\theta_{V}^{(\omega)})\right)
e^{-\sigma(\theta_{V}^{(\omega)})}\omega^{n}.
\end{align*}}\noindent
This is nothing but an obstruction for the existence of
multiplier Hermitian-Einstein metrics of type $\sigma$,
which is introduced by Futaki \cite{Fu02}.

\section{An $h$-modified extremal K\"{a}hler metric
and the properness of $K_{V}^{h}$}\label{exist/proper}
Based on the {\it existence\,/\,properness principle}
established by Darvas-Rubinstein \cite{DR17}
(see also Darvas's survey \cite{Dar19}),
we prove the first main result in this paper.
\begin{theorem}\label{main thm'}
There exists an $h$-modified extremal K\"{a}hler metric in
$\mathcal{H}_{V,0}$
if and only if
the $h$-modified Mabuchi functional $K_{V}^{h}$ is
$\mathrm{Aut}_{0}(M,V)$-invariant and coercive.
Moreover the condition of $\mathrm{Aut}_{0}(M,V)$-invariance can be
removed.
\end{theorem}

\begin{proof}
We apply \cite[Theorem~4.7]{Dar19} to the data
$\mathcal{R}=\mathcal{H}_{V,0}$, $d=d_{1}$, $F=K_{V}^{h}$
and $G=\mathrm{Aut_{0}}(M,V)$.
It is enough to check the above data enjoy the properties (P1)--(P6)
in \cite[Section~4.2]{Dar19}.

\begin{description}
\item[(P1)] This is stated in the preliminary.
\item[(P2)] This is due to Proposition~\ref{P2} stated below.
\item[(P3)] 
This is due to Theorem~\ref{regularity} stated below.
\item[(P4)] This is due to \cite[Lemma~5.9]{DR17}.
\item[(P5)] This is due to Theorem~\ref{BanMab} stated below.
\item[(P6)] This is the cocycle condition for $K_{V}^{h}$.
\end{description}

We argue that the condition of $G$-invariance can be removed.
For $X\in\mathrm{Lie}(\mathrm{Aut}_{0}(M,V))$,
let $g_{t}$ be a one-parameter subgroup generated by the flow of the
real part
$X_{\mathrm{Re}}:=\frac{1}{2}\left(X+\overline{X}\right)$.
For any $V_{\mathrm{Im}}$-invariant metric $\omega$,
we write $g_{t}^{*}\omega=\omega_{0}+\deldel \phi_{t}$.
Note that $g_{t}^{*}\omega$ is also $V_{\mathrm{Im}}$-invariant.
Since
$$
\frac{d}{dt}g_{t}^{*}\omega |_{t=0}=L_{X_{\mathrm{Re}}}\omega
=\deldel\mathrm{Re}(\theta_{X}^{(\omega)})
$$
where $\theta_{X}^{(\omega)}$ is the complex potential of $X$ with
respect to $\omega$,
then we have
\begin{eqnarray*}
\frac{d}{dt}K_{V}^{h}(\phi_{t})
&=&
-\int_{M}g_{t}^{*}\Big[\frac{d\phi}{dt}|_{t=0}\Big(S(\omega)
-\underline{S}-h(\theta_{V}^{(\omega)}
\Big)\omega^{n}\Big]\\
&=&
-\int_{M}\mathrm{Re}(\theta_{X}^{(\omega)})\Big(S(\omega)
-\underline{S}-h(\theta_{V}^{(\omega)})\Big)\omega^{n}
\end{eqnarray*}
which is independent of $t$.
Thus if $K_{V}^{h}$ is coercive, in particular bounded from below, then
$K_{V}^{h}(\phi_{t})$ is constant in $t$
which states that $K_{V}^{h}$ is $G$-invariant.
This completes the proof.
\end{proof}

\subsection{Compactness properties for sequences}
We argue the property (P2) in the proof of Theorem \ref{main thm'}.
\begin{proposition}\label{P2}
If $\{\phi_{i}\}_{i}\subset\mathcal{E}^{1}_{V,0}$ satisfies
$\lim_{i\to\infty}K^{h}_{V}(\phi_{i})
=\inf_{\mathcal{E}^{1}_{V,0}}K_{V}^{h}$
and 
$d_{1}(0,\phi_{i})< C$ for any $i$ for some $C>0$
then there exists a minimizer $\phi_{*}$ of $K_{V}^{1}$ in
$\mathcal{E}^{1}_{V,0}$
and a subsequence $\{\phi_{i_{k}}\}_{k}$ which converges to $\phi_{*}$
in $d_{1}$-topology.
\end{proposition}

\begin{proof}
Recall $K_{V}^{h}=H+\J_{-\mathrm{Ric}}+\J_{V}^{h}$.
Both functionals $\J_{-\mathrm{Ric}}$ and $\J_{V}^{h}$ are controlled by
the distance $d_{1}$
in general.
Thus $H(\phi_{i})$ is uniformly bounded.
The quantity $|\sup_{M}\phi_{i}|$ is also controlled by the distance
$d_{1}$
\cite[Lemma 4.4]{CC2}.
Then a compactness result in \cite[Theorem~5.6]{DR17} shows that
there exists $\phi_{*}\in\mathcal{E}^{1}_{V,0}$ and a subsequence
$\{\phi_{i_{k}}\}_{k}$
satisfying $d_{1}(\phi_{i_{k}},\phi_{*})\to 0$.
Since $K_{V}^{h}$ is $d_{1}$-lower semicontinuous, the limit $\phi_{*}$
is noting but a minimizer.
This completes the proof.
\end{proof}

The following stronger proposition than above one will be applied to
establish a geodesic stability in Section~\ref{gstability}.

\begin{proposition}\label{P2*}
If $\{\phi_{i}\}_{i}\subset\mathcal{E}^{1}_{V,0}$ satisfies
$K^{h}_{V}(\phi_{i})<C$ and 
$d_{1}(0,\phi_{i})< C$ for any $i$ for some $C>0$
then there exists $\phi\in\mathcal{H}^{1}_{V,0}$
and a subsequence $\{\phi_{i_{k}}\}_{k}$ which converges to $\phi$ in
$d_{1}$-topology.
\end{proposition}

\begin{proof}
By the same argument as above, we can use a compactness result in
\cite[Theorem~5.6]{DR17}
to conclude the proposition.
\end{proof}

\subsection{Regularity of minimizers}\label{Reg of mini}
We argue the property (P3) in the proof of Theorem \ref{main thm'}.
\subsubsection{A continuity path}
Let $T$ be the compact torus in $G:=\mathrm{Aut}_{0}(M,V)$
defined by the closure of the set of flows generated by $\mathrm{Im}V$.
For a $T$-invariant K\"{a}hler metrics $\omega$ and $\alpha$, 
we call $\omega$ a $\alpha$-twisted $h$-modified extremal metric if 
$S(\omega)-h(\theta_{V}^{(\omega)})-\Lambda_{\omega}\alpha=\text{const}$.
The following theorem
which generalizes Hashimoto's result \cite{Ha19},
gives an openness property for a continuity path used to prove the
regularity of minimizers.
(The following argument is essentially
due to He \cite[pp.\ 10--12]{He19pre},
where \cite{He19pre} is a preprint version of \cite{He19}.)

\begin{theorem}\label{Hashimoto}
Let $\omega$ and $\alpha$ be $T$-invariant K\"{a}hler metrics satisfying
$\mathrm{tr}_{\omega}\alpha=\text{const}$.
Then there exists a constant $r(\omega, \alpha)\in\mathbb{R}$ depending
only on $\omega$ and $\alpha$
such that for any $r\geqq{r(\omega,\alpha)}$, there exists
$\phi\in\mathcal{H}_{V}$ satisfying
$S(\phi)-h(\theta^{(\phi)}_{V})
-\Lambda_{\omega_{\phi}}(r\alpha)=\text{const}$.
\end{theorem}

\begin{proof}
At first, we take $T$-invariant K\"{a}hler metrics $\omega$ and $\alpha$
such that $\Lambda_{\omega}\alpha$ is not necessarily constant.
Set
$\tilde{S}(\omega)=
S(\omega)-h(\theta_{V}^{(\omega)})-\Lambda_{\omega}\alpha$.
Let $\psi$ be a variation of the potential function for $\omega$, 
that is, $\delta\omega=\deldel\psi$.
We follow Hashimoto's computation \cite{Ha19} to obtain
the linearization of the LHS of the above equation as follows.
\begin{eqnarray*}
\mathcal{L}(\psi)
&=&-\mathcal{D}^{*}_{\omega}\mathcal{D}_{\omega}\psi
+(\partial\tilde{S}(\omega),\dbar\psi)_{\omega}
+(\alpha,\deldel\psi)_{\omega}
+(\partial\Lambda_{\omega}\alpha,\dbar\psi)_{\omega}\\
&=&-\mathcal{D}^{*}_{\omega}\mathcal{D}_{\omega}\psi
+(\partial\tilde{S}(\omega),\dbar\psi)_{\omega}
+F_{\omega,\alpha}(\psi)
\end{eqnarray*}
where
$(\cdot,\cdot)_{\omega}$ is the point wise inner product on the space of
differential forms
defined by $\omega$,
$\mathcal{D}_{\omega}:C^{\infty}(M;\mathbb{R})\to
C^{\infty}(T^{1,0}M\otimes\Omega^{1,0}(M))$
is an operator defined by
$\mathcal{D}_{\omega}\psi=\dbar(\mathrm{grad}^{1,0}_{\omega}\psi)$
and $\mathcal{D}_{\omega}^{*}$ is its formal adjoint with respect to the
inner product.
Here we put
$F_{\omega,\alpha}(\psi):=(\alpha,\deldel\psi)_{\omega}
+(\partial\Lambda_{\omega}\alpha,\dbar\psi)_{\omega}$.
The kernel of $\mathcal{D}^{*}_{\omega}\mathcal{D}_{\omega}$ 
which is equal to the kernel of $\mathcal{D}_{\omega}$,
is the set of functions whose complex gradient is holomorphic
(see for example \cite[Section 4]{Szbook}).
The operator $F_{\omega,\alpha}$ is complex self-adjoint and second
order elliptic
and its kernel is the set of constant functions (\cite[Lemma~1]{Ha19}).
To apply the inverse function theorem we first observe
the following lemma:

%

\begin{lemma}\label{appro}
Let $\omega$ and $\alpha$ be $T$-invariant K\"{a}hler metrics satisfying 
$\Lambda_{\omega}\alpha=\text{const}$.
Then, for each $m\in\mathbb{N}$ there exists
$\phi_{1},\dots,\phi_{m}\in{C^{\infty}(M;\mathbb{R})}$
such that
$$\omega_{m}:=\omega+\deldel(r^{-1}\phi_{1}+\cdots+r^{-m}\phi_{m})$$
satisfies
$$
S(\omega_{m})-h(\theta^{(\omega_{m})}_{V})
-r\Lambda_{\omega_{m}}\alpha=
\text{const}+r^{-m}f_{m,r}$$
for a function $f_{m,r}$ with average $0$ with respect to $\omega_{m}$
which is bounded in $C^{\infty}(M;\mathbb{R})$ for any sufficiently
large $r$.
\end{lemma}

\begin{proof}[Proof of Lemma~\ref{appro}]
The proof is done by the induction.

By \cite[Lemma~2]{Ha19}, there exists a $T$-invariant function
$\phi_{1}\in C^{\infty}(M;\mathbb{R})$
such that
$-(S(\omega)-\underbar{S}-h(\theta^{(\omega)}_{V}))
=F_{\omega,\alpha}(\phi_{1})$,
that is,
$$
(S(\omega)-\underbar{S}-h(\theta^{(\omega)}_{V}))
+(\alpha, \deldel\phi_{1})_{\omega}=0.
$$
Since $F_{\omega,\alpha}$ is an elliptic operator, the uniform bounds of
derivatives of $\phi_{1}$
is guaranteed.
For large $r>0$, consider the K\"{a}hler metric
$\omega_{1}:=\omega+r^{-1}\deldel\phi_{1}$.
Note that
$\theta^{(\omega_{1})}_{V}=\theta^{(\omega)}_{V}+V(r^{-1}\phi_{1})$.
The Taylor expansion of the function $h$ and the uniform bound of
 $|\theta^{(\omega_{1})}_{V}|$ show
$h(\theta^{(\omega_{1})}_{V})
=h(\theta^{(\omega)}_{V})+\mathcal{O}(r^{-1})$.
Together with a computation in \cite[section~3.1]{Ha19}, we have
\begin{eqnarray*}
S(\omega_{1})-h(\theta^{(\omega_{1})}_{V})-r\Lambda_{\omega_{1}}\alpha
&=&
\text{const} + (S(\omega)-\underbar{S}-h(\theta^{(\omega)}_{V}))
+(\alpha,\deldel\phi_{1})_{\omega}+\mathcal{O}(r^{-1}) \\
&=& \text{const}+\mathcal{O}(r^{-1}).
\end{eqnarray*}
Therefore there exists a uniformly bounded smooth function $f_{1,r}$
such that
$$
S(\omega_{1})-h(\theta^{(\omega_{1})}_{V})-r\Lambda_{\omega_{1}}\alpha 
=\text{const}+r^{-1}f_{1,r}
\quad\text{and}\quad
\int_{M}f_{1,r}\omega_{1}^{n}=0.
$$

Assume we have functions $\phi_{1},\dots,\phi_{m-1}$ such that the
 statement in the lemma holds.
There exists a $T$-invariant function $\phi_{m}$ satisfying
$$
(S(\omega_{m-1})-\underbar{S}-h(\theta_{V}^{(\omega_{m-1})})
-r\Lambda_{\omega_{m-1}}\alpha+rc)
+r^{1-m}(\alpha,\deldel\phi_{m})_{\omega_{m-1}}=0
$$
where $c$ is the average of $\Lambda_{\omega_{m-1}}\alpha$ with respect
to $\omega_{m-1}$.
By the same argument as above, for large $r>0$, 
consider the K\"{a}hler metric
$\omega_{m}=\omega_{m-1}+r^{-m}\deldel\phi_{m}$
to compute
\begin{eqnarray*}
S(\omega_{m})-h(\theta^{(\omega_{m})}_{V})-r\Lambda_{\omega_{m}}\alpha
&=&
\text{const}
+(S(\omega_{m-1})-\underbar{S}-h(\theta_{V}^{(\omega_{m-1})})
-r\Lambda_{\omega_{m-1}}\alpha+rc)\\
&+&r^{1-m}(\alpha,\deldel\phi_{m})_{\omega_{m-1}}
+\mathcal{O}(r^{-m}) \\
&=&\text{const}+\mathcal{O}(r^{-m}).
\end{eqnarray*}
Therefore there exists a uniformly bounded smooth function $f_{m,r}$
such that
$$
S(\omega_{m})-h(\theta^{(\omega_{m})}_{V})-r\Lambda_{\omega_{m}}\alpha 
=\text{const}+r^{-1}f_{m,r}
\quad\text{and}\quad
\int_{M}f_{m,r}\omega_{m}^{n}=0.
$$
This completes the proof of Lemma~\ref{appro}.
\end{proof}

Take a smooth function $h_{m,r}$ satisfying
$\square_{\omega_{m}}h_{m,r}=f_{m,r}$ and
$\int_{M}h_{m,r}\omega_{m}^{n}=0$.
Since $\omega_{m}=\omega+\mathcal{O}(r^{-1})$, the elliptic theory gives
the uniform bound of $h_{m,r}$.
Define $\alpha_{m}=\alpha+r^{-m-1}\deldel h_{m,r}$ to satisfy
$$
S(\omega_{m})-h(\theta^{\omega_{m}}_{V})-r\Lambda\omega_{m}\alpha_{m}
=\text{const}.
$$

In order to apply the inverse function theorem for Banach spaces,
let $W^{k,2}_{0,T}$ be the $T$-invariant Sobolev space on $M$
such that the average of each function is $0$ with respect to $\omega$
and let $\Omega_{T}^{1,1}$ be the set of $T$-invariant real
$(1,1)$-forms on $M$
completed by the $W^{k,2}$-Sobolev norm.
Here $k$ is sufficiently large and all Sobolev norms are defined by
$\omega$.
Consider Banach spaces $B_{1}:=\Omega_{T}^{1,1}\times W^{k+4,2}_{0,T}$
and $B_{2}:=\Omega_{T}^{1,1}\times W^{k,2}_{0,T}$
and take an open set
$U:=\Set{(\eta,\phi)\in{B_{1}}|\omega_{m}+\deldel\phi>0}$.
For $r>0$ we define a map $T_{r}:U\to B_{2}$ as follows.
$$
T_{r}(\eta,\phi)
=(\alpha_{m}+\eta, 
S(\omega_{m,\phi})-h(\theta^{\omega_{m,\phi}}_{V})
-r\Lambda_{\omega_{m,\phi}}(\alpha_{m}+\eta))
$$
where $\omega_{m,\phi}=\omega_{m}+\deldel\phi$.
The point $0\in U$ corresponds to 
the $r\alpha_{m}$-twisted $h$-modified extremal metric $\omega_{m}$.
Note that $T_{r}(0,0)=(\alpha_{m}, \underbar{S}-rc)$
where $c$ is the average of $\Lambda_{\omega_{m}}\alpha_{m}$ which is a
topological constant.
We need to find $(\eta,\phi)\in U$ satisfying
$T_{r}(\eta,\phi)=(\alpha,\underbar{S}-rc)$
by applying the inverse function theorem.
We write $\Lambda_{m}$ for $\Lambda_{\omega_{m}}$,
$\mathcal{D}_{m}^{*}\mathcal{D}_{m}$ for 
$\mathcal{D}_{\omega_{m,}}^{*}\mathcal{D}_{\omega_{m}}$ and 
$F_{m}$ for $F_{\omega_{m},\alpha_{m}}$.
The linearization of $T_{r}$ at $0\in U$ is given by
$$
DT_{r}|_{0}(\tilde{\eta},\tilde{\phi})=(\tilde{\eta},\tilde{\phi})
\begin{pmatrix}
1&-r\Lambda_{m}\\
0& -\mathcal{D}^{*}_{m}\mathcal{D}_{m}+rF_{m}
\end{pmatrix}.
$$
This linearized operator is exactly same as Hashimoto's one for a
twisted CSCK metric \cite{Ha19}.
Therefore the property $\alpha_{m}=\alpha+\mathcal{O}(r_{-m-1})$ 
and an inverse function theorem argument by
Hashimoto \cite[section3.2]{Ha19} show theorem~\ref{Hashimoto}.
\end{proof}

Similar to \cite[Corollary 1]{Ha19}, we have the following which shows
the openness
property for a continuity path used to prove the regularity of
minimizer.

\begin{corollary}\label{openness}
Let $\omega$ is an $\alpha$-twisted $h$-modified extremal metric such
that $\omega$ and $\alpha$ are
$T$-invariant K\"{a}hler metrics.
If $\tilde{\alpha}$ is a $T$-invariant K\"{a}hler metric which is close
to
$\alpha$ in the $C^{\infty}$-norm,
then there exists an $\tilde{\alpha}$-twisted $h$-modified extremal
metric in $[\omega]$ which is $T$-invariant.
\end{corollary}

\subsubsection{Proof of the regularity}

\begin{lemma}\label{existence of h}
Suppose there exists a minimizer $\phi_{*}$ of $K_{V}^{h}$ in
$\mathcal{E}^{1}_{V}$. 
Then there exists a smooth $h$-modified extremal K\"{a}hler metric in
$\mathcal{H}_{V}$.
\end{lemma}

\begin{proof}
We follow the argument in \cite[Lemma 5.1]{CC2}.
Consider the continuity path
\begin{equation}\label{eq-t}
t(S(\phi)-\underbar{S}-h(\theta^{\phi}_{V}))
-(1-t)(\Lambda_{\omega_{\phi}}\omega_{0}-n)=0
\end{equation}
where $\omega_{0}$ and $\omega_{\phi}$ are $T$-invariant
and $t\in [0,1]$.
Theorem~\ref{Hashimoto} and Corollary~\ref{openness} show
the set of $t$ such that a solution of \eqref{eq-t} exists
is open in $[0,1]$.
A solution $\phi(t)$ of the equation \eqref{eq-t} is a minimizer of the
functional
$$
\tilde{K}_{t}:=tK^{h}_{V}+(1-t)\J.
$$
This is because $K^{h}_{V}$ is convex and $\J$ is affine along $C^{1,1}$
geodesic segments in $\mathcal{H}_{V}$.

By the assumption, $K_{V}^{h}$ is bounded from below.
The inequality \eqref{IJ} and an estimate in
\cite[Proposition~5.5]{DR17} shows
\begin{equation}\label{Jproper}
\J(\phi)\geqq{C^{-1}d_{1}(0,\phi)-C}
\end{equation}
for any $\phi\in\mathcal{H}_{0}.$
Thus $\tilde{K}_{t}(\phi)\geqq{C^{-1}(1-t)d_{1}(0,\phi)-C}$
for any $\phi\in\mathcal{H}_{0}$.
Let $\phi(t)\in\mathcal{H}_{V,0}$ be a solution of \eqref{eq-t}.
Since $\phi(t)$ is a minimizer of $\tilde{K}_{t}$, we have
$d_{1}(0,\phi(t))<C((1-t)^{-1}+1)$.
Since $|\J_{-\mathrm{Ric}}(\phi)|$ and $|\J^{h}_{V}(\phi)|$ are
bounded by $Cd_{1}(0,\phi)$ in general where $C>0$ is a uniform constant,
by definition of $\tilde{K}_{t}$,  we have an inequality
$$
H(\phi(t))=
\vol^{-1}\int_{M}\log
\Big(\frac{\omega_{\phi(t)}^{n}}{\omega_{0}^{n}}\Big)
\omega_{\phi(t)}^{n}\leq Cd_{1}(0,\phi(t))+C<C((1-t)^{-1}+1)
$$
for each $t\in(0,1)$.
A priori estimate by Chen-Cheng \cite[Theorem~1.2]{CC1} shows the
equation \eqref{eq-t} has a smooth solution for any $t\in (0,1)$.

Take a sequence $t_{i}\to 1$, and take the solution $\phi_{i}$ of the
equation \eqref{eq-t} for $t=t_{i}$ with
the normalization $E(\phi_{i})=0$.
An argument in \cite[Corollary 4.5]{CC2} shows $\phi_{i}$ minimizes
$t_{i}K_{V}^{h}+(1-t_{i}) \J$
on the completion $\mathcal{E}_{V}^{1}$,
which is based on the fact that $tK_{V}^{h}+(1-t)\J$ is convex along
finite energy geodesic segments (cf. \cite[Theorem~4.7]{BDL20}).
Since
$$
t_{i}K_{V}^{h}(\phi_{*})+(1-t_{i})\J(\phi_{i})\leqq
t_{i}K_{V}^{h}(\phi_{i})+(1-t_{i})\J(\phi_{i})\leqq
t_{i}K_{V}^{h}(\phi_{*})+(1-t_{i})\J(\phi_{*}),
$$
we have $\J(\phi_{i})\leq\J(\phi_{*})$ for any $i$.
Together with \eqref{Jproper}, $\sup_{i}d_{1}(0,\phi_{i})<+\infty$.
By the same inequality as above (see also \cite[Proposition 3.1]{He19}),
we have a uniform entropy bound
$\sup_{i}H(\phi_{i})<+\infty$.
A priori estimate by Chen-Cheng \cite[Theorem~1.2]{CC1} again shows
a subsequence of
$\phi_{i}$ converges to a smooth $h$-modified extremal K\"{a}hler metric
$\phi_{\infty}$.
This completes the proof of the Lemma.
\end{proof}

\begin{theorem}\label{regularity}
Let $\phi_{*}$ be the minimizer of $K_{V}^{h}$ in
$\mathcal{E}_{V}^{1}$.
Then $\phi_{*}$ is smooth and defines a smooth $h$-modified extremal
K\"{a}hler metric.
\end{theorem}

\begin{proof}
We follow the argument in \cite[Theorem 5.1]{CC2} and \cite[Theorem 3.6]{He19}.
By \cite[Lemma~3.8]{He19}, we first take a sequence
$\phi_{i}\in\mathcal{H}_{V}$ such that
$\lim_{i\to\infty}d_{1}(\phi_{i},\phi_{*})=0$ and
$\lim_{i\to\infty}H(\phi_{i})=H(\phi)$.
Then $\lim_{i\to\infty}K_{V}^{h}(\phi_{i})=K_{V}^{h}(\phi_{*})$ 
by the $d_{1}$-continuity.
For each $i$, consider the continuity path
\begin{equation}\label{continuity path}
t(S(\phi)-h(\theta_{V}^{\phi})-\underbar{S})
-(1-t)(\Lambda_{\omega_{\phi}}\omega_{\phi_{i}}-n)=0
\end{equation}
for $t\in (0,1]$.
Theorem~\ref{Hashimoto} and Corollary~\ref{openness} show the set of $t$
such that a solution of \eqref{continuity path} exists is open in
$(0,1]$.
Consider the Aubin's functionals with a reference metric
$\omega_{\phi_{i}}$.
Namely we consider
\begin{align*}
&I_{i}(\phi):=\vol^{-1}
\int_{M}\phi(\omega_{\phi_{i}}^{n}-\omega_{\phi}^{n}),\\
&J_{i}(\phi):=
\vol^{-1}\int_{M}\phi\omega_{\phi_{i}}^{n}
-\frac{\vol^{-1}}{n+1}\int_{M}\phi\sum_{j=0}^{n}
\omega_{\phi_{i}}^{n}\wedge\omega_{\phi}^{n}
\end{align*}
and $\J_{i}:=I_{i}-J_{i}$.
By the same argument as in Lemma~\ref{existence of h}, there exists a
unique smooth solution
$\phi_{i}^{t}\in\mathcal{H}_{V,0}$ for any $t\in (0,1)$ which is a
minimizer of
$\tilde{K}_{t}^{i}:=tK_{V}^{h}+(1-t)\J_{i}$ on $\mathcal{E}_{V}$.
Thus for any $\phi\in\mathcal{E}_{V}$,
\begin{equation}\label{ineq10}
tK_{V}^{h}(\phi_{i}^{t})+(1-t)\J_{i}(\phi_{i}^{t})\leqq
tK_{V}^{h}(\phi)+(1-t)\J_{i}(\phi).
\end{equation}
Since it is easy to see that $\phi_{i}$ minimizes $\J_{i}$, we have 
$0=\J_{i}(\phi_{i})\leqq\J_{i}(\phi_{i}^{t}).$
Together with \eqref{ineq10}, we have
\begin{equation}\label{K-bound}
K_{V}^{h}(\phi_{i}^{t})\leqq{K_{V}^{h}(\phi_{i})}.
\end{equation}
Since $\phi_{*}$ minimizes $K_{V}^{h}$, the inequality \eqref{ineq10}
gives
$\J_{i}(\phi_{i}^{t})\leqq\J_{i}(\phi_{*}).$
Note the classical inequality $I_{i}\leqq(n+1)\J_{i}\leqq{nI_{i}}$ and
an inequality
$I_{i}(\phi_{*})\leqq{Cd_{1}(\phi_{i},\phi_{*})}$ proved by 
Darvas \cite{Dar15}.
Summing these inequalities up, we have
$I_{i}(\phi_{i}^{t})\leqq Cd_{1}(\phi_{i},\phi_{*})$
where the constant $C$ is independent of $t$.
Since $\lim_{i\to\infty}d_{1}(\phi_{i},\phi_{*})=0$,
we have $\lim_{i\to\infty}I_{i}(\phi_{i}^{t})=0$
which is uniform in $t\in(0,1)$.
By \cite[Theorem~1.8]{BBEGZ19} and the inequality
$|I_{\omega_{0}}(\phi_{i})-I_{\omega_{0}}(\phi_{*})|\leqq Cd_{1}(\phi_{i},\phi_{*})$,
we have
$I_{\omega_{0}}(\phi_{i}^{t})
\leqq{C(I_{\omega_{0}}(\phi_{i})+I_{i}(\phi_{i}^{t}))}\leqq{C}$
which is uniform in $t\in(0,1)$ and $i\gg1$.
By \cite[Proposition~5.5]{DR17}, we thus have the distance bound
$d_{1}(0,\phi_{i}^{t})\leqq{CI_{\omega_{0}}(\phi_{i}^{t})+C}\leqq{C}$
which is uniform in $t\in(0,1)$ and $i\gg1$.
The inequality \eqref{K-bound} and this distance bound show that the
entropy $H(\phi_{i}^{t})$ has uniform upper bound in $t\in(0,1)$ and $i\gg1$. 
By a priori estimate by Chen-Cheng \cite[Theorem~1.2]{CC1},
all derivatives of $\phi_{i}^{t}$ are uniformly bounded in $t\in(0,1)$ and $i\gg1$.
Thus, when $t\to 1$, a subsequence of
the potential $\phi_{i}^{t}$ converges smoothly to
$u_{i}\in\mathcal{H}_{V,0}$
solving the equation \eqref{continuity path}, that is,
the h-modified extremal K\"{a}hler metric equation
$S(u_{i})-\underbar{S}-h(\theta_{V}^{(u_{i})})=0$.

Note that $d_{1}(0,u_{i})$ is uniformly bounded.
We can take a subsequence $\{u_{i}\}_{i}$ such that this converges smoothly to a
smooth h-modified extremal K\"{a}hler metric $u\in\mathcal{H}_{V,0}$.
By $\lim_{i\to\infty}I_{i}(u_{i})=0$ and an argument in \cite[Lemma~5.7]{CC2},
the difference between $u$ and
$\phi_{*}$ is constant.
This completes the proof.
\end{proof}

\subsection{Bando-Mabuchi type uniqueness theorem}
We argue the property (P5) in the proof of Theorem \ref{main thm'}.
\begin{theorem}\label{BanMab}
For any two $h$-modified extremal K\"{a}hler metrics
$\omega_{\phi_{0}},\omega_{\phi_{1}}$ in $\mathcal{H}_{V}$,
there exists $g\in\mathrm{Aut}_{0}(X,V)$ such that
$\omega_{\phi_{0}}=g^{*}\omega_{\phi_{1}}$.
\end{theorem}

\begin{proof}
We follow an argument of Berman-Berndtsson \cite{BB17}.
Let $T$ be the compact torus
defined by the closure of the set of flows generated by $\mathrm{Im}V$,
and let $\mu$ be a $T$-invariant smooth volume form on $M$ with
$\int_{M}d\mu=\Vol$. 
For any $\phi\in\mathcal{H}_{V}$, define
$$
F_{\mu}(\phi):=\int_{M}\phi d\mu-E(\phi):=I_{\mu}(\phi)-E(\phi).
$$
Consider the functional $K_{V}^{h}+sF_{\mu}$ for small $s>0$ to define
the derivative
$F_{V}(\phi,s):=d(K_{V}^{h}+sF_{\mu})|_{\phi}$
at $\phi\in\mathcal{H}_{V}$.

The space of K\"{a}hler metrics $\mathcal{H}_{V}$ has a natural
$L^{2}$-metric
(called the Mabuchi metric)
defined by $\int_{M}v_{0}v_{1}\omega_{\phi}^{n}/\Vol$ for tangent
vectors $v_{0},v_{1}$
in $T_{\phi}\mathcal{H}_{V}$.
This induces the connection $D$ on the space of $1$-forms on
$\mathcal{H}_{V}$.
We regard the differentials $dK_{V}^{h}$ and $dF_{\mu}$ 
of the functionals $K_{V}^{h}$ and $F_{\mu}$ 
as $1$-forms on $\mathcal{H}_{V}$ and consider the equation
\begin{equation}\label{1-form}
-D_{v_{0}}dK_{V}^{h}|_{\phi_{0}}=dF_{\mu}|_{\phi_{0}}
\end{equation}
for unknown $v_{0}\in T_{\phi_{0}}\mathcal{H}_{V}$.

We put
$$
\mathfrak{g}_{\omega_{0}, V}:=
\Set{v\in\mathrm{Lie}(\mathrm{Aut}_{0}(M,V))\,|\,
i_{v}\omega_{0}=\sqrt{-1}\,\dbar{f} 
\text{ where }f\in C^{\infty}(M;\mathbb{R})}
$$
and
$\Gamma_{\omega_{0},V}:=\exp(\mathfrak{g}_{\omega_{0},V})$.
In view of an argument in \cite[Proposition~4.7]{BB17},
we can assume that the $h$-modified extremal K\"{a}hler metric
$\phi_{0}$
minimizes
$F_{\mu}$
on the orbit
$(\Gamma_{\omega_{0},V})\cdot[\phi_{0}]\subset\mathcal{H}_{V}$
by the replacing $\omega_{\phi_{0}}$ by $g_{0}^{*}\omega_{\phi_{0}}$
for some $g_{0}\in\Gamma_{\omega_{0},V}$.
Since the Hessian of $\mathbb{J}_{V}^{H}$ vanishes on
$T_{\phi_{0}}\mathcal{H}_{V}$,
an argument in
\cite[Proposition~4.3, Proof of Theorem~4.4
and pp.\ 1192-1193]{BB17}
shows that the equation \eqref{1-form} has
a $T$-invariant solution $v_{0}$.
This implies
\begin{equation}\label{F0}
F_{V}(\phi_{0}+sv_{0}, s)=\mathcal{O}(s^{2}),
\end{equation}
since for a smooth curve $\psi_{s}$ of continuous function,
$$
\frac{d}{ds}\Big|_{0}\langle
F_{V}(\phi_{0}+sv_{0}, s),\psi_{s}\rangle
=\langle
D_{v_{0}}dK_{V}^{h}|_{\phi_{0}},\psi_{0}\rangle
+\langle
dK_{V}^{h}|_{\phi_{0}},D_{v_{0}\psi_{s}}\rangle
+\langle
dF_{\mu}|_{\phi_{0}},\psi_{0}\rangle=0
$$
where $\langle\cdot,\cdot\rangle$ is the pairing for $1$-forms.
Similarly, under replacing $\omega_{\phi_{1}}$
by $g_{1}^{*}\omega_{\phi_{1}}$
for some $g_{1}\in\mathrm{Aut}_{0}(X,V)$,
there exists $T$-invariant $v_{1}\in T_{\phi_{1}}\mathcal{H}_{V}$
satisfying
\begin{equation}\label{F1}
F_{V}(\phi_{1}+sv_{1},s)=\mathcal{O}(s^{2}).
\end{equation}

Consider the $C^{1,1}$-geodesic $\phi_{t}^{s}$ connecting
$\phi_{0}+sv_{0}$ and $\phi_{1}+sv_{1}$.
Note that $\phi_{t}^{s}$ is also $T$-invariant.
Finally, an argument in
\cite[Proposition~4.1 and Proof of Theorem~4.4]{BB17}
together with
the convexity of $t\mapsto K_{V}^{h}(\phi_{t}^{s})$, the linearity of
$t\mapsto E(\phi_{t}^{s})$ and
the properties \eqref{F0} and \eqref{F1} show that there exists a
uniform constant $C>0$ such that
$d_{2}(\phi_{0}^{s},\phi_{1}^{s})^{2}\leqq{Cs}$, which implies
$\omega_{\phi_{0}}=\omega_{\phi_{1}}$.
\end{proof}

\section{An $h$-modified extremal K\"{a}hler metrics
and the geodesic stability}
We argue the characterization for the existence of an $h$-modified
extremal K\"{a}hler metric
in terms of geodesic stability conditions
which has independent interest to study of the metric.
\label{gstability}
\subsection{Uniform geodesic stability}
We first characterize the existence of an $h$-modified extremal
K\"{a}hler
metric as a uniform geodesic stability condition.
Let $G:=\mathrm{Aut}_{0}(M,V)$.

\begin{definition}
Fix $\psi_{0}\in\mathcal{H}_{V,0}$ with $K_{V}^{h}(\psi_{0})<+\infty$.
We say $(M,\Omega,V)$ is
{\it{uniform geodesic stable}} at $\psi_{0}$ if
there exists $C>0$ such that for any geodesic ray
$[0,+\infty)\ni t \to \psi_{t}\in\mathcal{E}^{1}_{V,0}$ starting
at $\psi_{0}$, we have
$$
\lim_{t\to+\infty}\frac{K^{h}_{V}(\psi_{t})-K^{h}_{V}(\psi_{0})}{t}
\geqq{C\limsup_{t\to+\infty}\frac{d_{1,G}(\psi_{t},\psi_{0})}{t}}.
$$
\end{definition}

\begin{theorem}\label{uniform stable thm}
Fix $\psi_{0}\in\mathcal{E}_{V,0}^{1}$ with
$K_{V}^{h}(\psi_{0})<+\infty$.
There exists an $h$-modified extremal K\"{a}hler metric in
$\mathcal{H}_{V,0}$
if and only if
the $h$-modified Mabuchi functional $K^{h}_{V}$ is $G$-invariant and
$(M,\Omega,V)$ is uniform geodesic stable at
$\psi_{0}\in\mathcal{E}_{V,0}^{1}$ with
$K_{V}^{h}(\psi_{0})<+\infty$.
Moreover, the condition of $G$-invariance can be removed.
\end{theorem}

\begin{proof}
Following \cite[Theorem~4.7]{Dar19}, it suffice to establish properties
(P1), (P3), (P4), (P5), (P6) and (P2)*.
The last property (P2)* is due to Proposition~\ref{P2*}.
The other properties are explained in the proof of
Theorem~\ref{main thm'}.
We argue that 
the condition of $G$-invariance can be removed
in the next lemma.
This completes the proof.
\end{proof}

Note that the limit in the left hand side of the above inequality exists
because of the convexity of $K^{h}_{V}$.

\begin{lemma}\label{Kbounded}
Fix $\psi_{0}\in\mathcal{H}_{V,0}$ with $K_{V}^{h}(\psi_{0})<+\infty$.
If for any geodesic ray
$[0,+\infty)\ni{t}\to\psi_{t}\in\mathcal{E}^{1}_{V,0}$
starting at $\psi_{0}$, we have
$$
\lim_{t\to+\infty}
\frac{K^{h}_{V}(\psi_{t})-K^{h}_{V}(\psi_{0})}{t}\geqq{0},
$$
then the $h$-modified Mabuchi functional $K_{V}^{h}$ is $G$-invariant.
\end{lemma}

\begin{proof}
We follow the proof of \cite[Lemma~7.1]{CC2}.
Take any $\phi\in\mathcal{H}_{V,0}$ and any one parameter subgroup 
$[0,+\infty)\ni{t}\mapsto{g_{t}}\in\mathrm{Aut}_{0}(M,V)$
generated by $X_{\mathrm{Re}}$ where $X\in\mathrm{Aut}_{0}(M,V)$.
Define $\phi_{t}:=g_{t}[\phi]\in\mathcal{H}_{V,0}$.
It suffice to show $\frac{d}{dt}K_{V}^{h}(\phi_{t})=0$.
Since the same argument as in the proof of Theorem~\ref{main thm'} shows
$\frac{d}{dt}K_{V}^{h}(\phi_{t})$ is independent of $t$,
we can assume $\frac{d}{dt}K_{V}^{h}(\phi_{t})(=:a)<0$
without loss of generality.

The same argument as in the proof of \cite[Lemma~7.1]{CC2} shows
that there exists a constant $C>0$ such that
for a sequence $t_{k}\nearrow+\infty$ and for the unit speed finite
energy
geodesic segments
$\rho_{k}(s):[0,d_{1}(\psi_{0},\phi_{t_{k}})]
\to\mathcal{E}_{V,0}^{h}$ 
connecting $\psi_{0}$ and $\phi_{t_{k}}$, we have
$$
K_{V}^{h}(\rho_{k}(s))\leqq
\mathrm{max}\{K_{V}^{h}(\psi_{0}),K_{V}^{h}(\phi)\}
+\frac{sat_{k}}{Ct_{k}+d_{1}(\psi_{0},\phi)}
$$
for any $s\in [0,d_{1}(\psi_{0},\phi_{t_{k}})]$.
Here we used the lower semi continuity of $K_{V}^{h}$ and the convexity
of $K_{V}^{h}$ along $\rho_{k}$.
Thus for each fixed $s$, the functional $K_{V}^{h}(\rho_{k}(s))$ is
bounded above 
which is uniform in $k$.
As discussed in the proof of Proposition~\ref{P2}, 
for each fixed $s$, a compactness result in
\cite[Theorem~5.6]{DR17} can be applied to show that
a subsequence $\{\rho_{k_{l}}(s)\}_{l}\subset\mathcal{E}_{V,0}^{1}$
converges in $d_{1}$-distance to an element in $\mathcal{E}_{V,0}^{1}$.
The same argument which is used Cantor's diagonal sequence argument
as in the proof of \cite[Lemma~6.5]{CC2} concludes that
there exists a subsequence $\{k_{l}\}_{l}$ such that for any
$s\geqq{0}$,
$\rho_{k_{l}}(s)$ converges in $d_{1}$-distance to a unit speed locally
finite energy
geodesic ray $\rho_{\infty}(s)$ starting at $\psi_{0}$.
By the lower semicontinuity of $K_{V}^{h}$,
$$
K_{V}^{h}(\rho_{\infty}(s))\leqq\liminf_{l}K_{V}^{h}(\rho_{k_{l}}(s))
\leqq\mathrm{max}\{K_{V}^{h}(\phi),K_{V}^{h}(\psi_{0})\}+\frac{sa}{C}
$$
to conclude
$$
\lim_{s\to+\infty}
\frac{K_{V}^{h}(\rho_{\infty}(s)-K_{V}^{h}(\psi_{0})}{s}
\leqq\frac{a}{C}<0
$$
which is a contradiction.
\end{proof}

\subsection{Geodesic stability}
Next we characterizes the existence of an $h$-modified extremal
K\"{a}hler metric as a geodesic stability condition.

\begin{definition}
Fix $\psi_{0}\in\mathcal{E}_{V,0}^{1}$ with
$K_{V}^{h}(\psi_{0})<+\infty$.
We say $(M,\Omega, V)$ is {\it{geodesic stable}} at $\psi_{0}$ if
for any locally finite energy geodesic ray
$\psi_{t}:[0,+\infty)\to\mathcal{E}_{V,0}^{1}$
starting at $\psi_{0}$ with unit speed,
exactly one of the following alternative holds.
\begin{enumerate}
\item $\lim_{t\to+\infty}
(K_{V}^{h}(\psi_{t})-K_{V}^{h}(\psi_{0}))/t>{0},$
\item $\lim_{t\to+\infty}
(K_{V}^{h}(\psi_{t})-K_{V}^{h}(\psi_{0}))/t=0$
and $\psi_t$ is parallel to another
geodesic ray
$\tilde{\psi_{t}}: [0,+\infty)\to\mathcal{E}_{V,0}^{1}$ which is
generated from a holomorphic
vector field in $\mathrm{Lie}(G)$. 
Here we say that $\psi_{t}$ is parallel to $\tilde{\psi_{t}}$ if 
$\sup_{t>0}d_{1}(\psi_{t},\tilde{\psi_{t}})<+\infty$.
\end{enumerate}
\end{definition}


\begin{proposition}\label{bounded of K}
If $(M,\Omega,V)$ is geodesic stable at
$\psi_{0}\in\mathcal{E}_{V,0}^{1}$
with $K_{V}^{h}(\psi_{0})<+\infty$, then the $h$-modified Mabuchi
functional is bounded from below.
\end{proposition}

\begin{proof}
We follow the proof of \cite[Proposition~7.3]{CC2}.
In order to imply a contradiction, suppose that there exists a sequence
$\{\tilde{\phi_{i}}\}\subset\mathcal{E}_{V,0}^{1}$ satisfying
 $K_{V}^{h}(\tilde{\phi_{i}})\to-\infty$.
Take $g_{i}\in G:=\mathrm{Aut}_{0}(M,V)$ to define
$\phi_{i}:=g_{i}[\tilde{\phi_{i}}]\in\mathcal{E}_{V,0}^{1}$ with the
property
$d_{1,G}(\psi_{0},\tilde{\phi_{i}})\leqq{d_{1}(\psi_{0},\phi_{i})}
\leqq{d_{1,G}(\psi_{0},\tilde{\phi_{i}})+1}$.
In view of Lemma~\ref{Kbounded}, $K_{V}^{h}(\phi_{i})\to-\infty$.

If $\sup_{i}d_{1}(\psi_{0},\phi_{i})<+\infty$, then the same compactness
argument as in
Proposition~\ref{P2} shows a subsequence $\{\phi_{i_{k}}\}_{k}$
converges to
$\phi_{\infty}\in\mathcal{E}_{V,0}^{1}$ in $d_{1}$-distance.
It follows from the lower semicontinuity of $K_{V}^{h}$ that
$K_{V}^{h}(\phi_{\infty})\leqq
\liminf_{k\to\infty}K_{V}^{h}(\phi_{i_{k}})=-\infty$
which is a contradiction. 

Thus we focus on the case where
$\sup_{i}d_{1}(\psi_{0},\phi_{i})=+\infty$.
We can assume $d_{1}(\psi_{0},\phi_{i})\to+\infty$.
Let $\rho_{i}:[0,d_{1}(\psi_{0},\phi_{i})]\to\mathcal{E}_{V,0}^{1}$ be
the unit speed geodesic
segment connecting $\psi_{0}$ and $\phi_{i}$.
By the convexity of $K_{V}^{h}$ along $\rho_{i}$, we have
$K_{V}^{h}(\rho_{i}(t))\leqq
\mathrm{max}\{K_{V}^{h}(\psi_{0}),K_{V}^{h}(\phi_{i})\}$
for $t\in[0,d_{1}(\psi_{0},\phi_{i})]$.
Since $K_{V}^{h}(\rho_{i}(t))$ is bounded from above for each fixed
$t>0$, the same argument as
in \cite[Lemma~6.5]{CC2} again shows that the subsequence of the
geodesic segments
converges to a unit speed locally finite geodesic ray 
$\rho_{\infty}:[0,+\infty)\to\mathcal{E}_{V,0}^{1}$ starting
at $\psi_{0}$.
By the lower semicontinuity of $K_{V}^{h}$, the functional $K_{V}^{h}$
is still bounded above
on $\rho_{\infty}$.
Then the geodesic stability condition implies that $\rho_{\infty}$ is
parallel to a geodesic ray
generated from a holomorphic vector field.
This concludes $d_{1,G}(\rho_{\infty}(0),\rho_{\infty}(t))$ is bounded
uniformly in $t$.

On the other hand, by \cite[Lemma 7.4]{CC2}, we have
$$
d_{1,G}(\psi_{0},\rho_{i}(t))
\geqq{d_{1}(\psi_{0},\rho_{i}(t))-1}=t-1
$$
for $t\in [0,d_{1}(\psi_{0},\phi_{i})]$ to conclude
$$
d_{1,G}(\psi_{0},\rho_{\infty}(t))
\geqq{d_{1,G}(\psi_{0},\rho_{i}(t))
-d_{1,G}(\rho_{i}(t),\rho_{\infty}(t))}
\geqq{t-1-d_{1}(\rho_{i}(t),\rho_{\infty}(t))}.
$$
Taking $i\to\infty$, we have a contradiction to boundedness of 
$d_{1,G}(\rho_{\infty}(0),\rho_{\infty}(t))$.
\end{proof}

\begin{theorem}\label{geod stability}
Fix $\psi_{0}\in\mathcal{E}_{V,0}^{1}$ with
$K_{V}^{h}(\psi_{0})<+\infty$.
There exists an $h$-modified extremal K\"{a}hler metrics in
$\mathcal{H}_{V,0}$ if and only if
$(M,\Omega,V)$ is geodesic stable at
$\psi_{0}\in\mathcal{E}_{V,0}^{1}$.
\end{theorem}

\begin{proof}
In order to show the ``only if'' part, we follow an argument in
\cite[pp.\ 992--993]{CC2}.
Assume that the reference metric $\omega_{0}$ is an $h$-modified
extremal K\"{a}hler metric.
By Theorem~\ref{main thm'}, the functional $K_{V}^{h}$ is coercive.

Let $\rho:[0,+\infty)\to\mathcal{E}_{V,0}^{1}$ be a unit speed geodesic
ray starting at $\psi_{0}$.
If $K_{V}^{h}$ is unbounded from above along $\rho$, then the convexity
of $K_{V}^{h}$ shows
$\lim_{t\to+\infty}K_{V}^{h}(\rho(t))/t>0$.
Thus we argue that $\rho$ is parallel to a geodesic ray
coming from a holomorphic vector field, provided that $K_{V}^{h}$ is
bounded from above on $\rho$.

Let $\{t_{k}\}_{k}\subset\mathbb{R}$ be a sequence such that $t_{k}>0$
and $t_{k}\to+\infty$.
Let $r_{k} : [0, d_{1}(0,\rho(t_{k}))]]\to\mathcal{E}_{V,0}^{1}$ be the
unit speed finite
energy geodesic segment connecting $0$ and $\rho(t_{k})$.
We use the fact that the convexity of $K_{V}^{h}$ along $r_{k}$ and that
a $h$-modified extremal K\"{a}hler metric minimizes $K_{V}^{h}$ to
conclude
for any $t\in [0,d_{1}(0,\rho(t_{k}))]$,
$$
K_{V}^{h}(r_{k}(t))\leqq\Big(1-\frac{t}{d_{1}(0,\rho(t_{k}))}\Big) 
\inf_{\mathcal{E}_{V,0}^{1}}K_{V}^{h}
+\frac{t}{d_{1}(0,\rho(t_{k}))}\sup_{t\geqq{0}}K_{V}^{h}(\rho(t)).
$$
Since $K_{V}^{h}$ is uniformly bounded from above, we can apply the same
argument
as in \cite[Lemma~6.5]{CC2} to show that
for a subsequence $\{k_{l}\}_{l}$ and for any $t\geqq{0}$, the ray
$r_{k_{l}}(t)$ converges to
a locally finite energy geodesic ray $r_{\infty}(t)$ with unit speed.
Then by the above inequality and the lower semicontinuity of
$K_{V}^{h}$, we have
$$
K_{V}^{h}(r_{\infty}(t))\leqq
\liminf_{k_{l}}K_{V}^{h}(r_{k_{l}}(t))\leqq
\inf_{\mathcal{E}_{V,0}^{1}}K_{V}^{h},
$$
for any $t>0$.
Since $r_{\infty}(t)$ is thus an $h$-modified extremal K\"{a}hler metric
for any $t\geqq{0}$,
Theorem~\ref{BanMab} shows the ray $r_{\infty}$ is generated by a
holomorphic vector field.
Moreover by \cite[Lemma~7.6]{CC2}, we know that $\rho$ and $r_{\infty}$
are parallel.
This completes the proof of the ``only if'' part.

In order to show the ``if'' part,
we follow an argument in \cite[p. 992]{CC2}.
Recall the continuity path \eqref{eq-t} to construct an $h$-modified
extremal K\"{a}hler metric.
By Proposition~\ref{bounded of K} and a argument in the proof of
Lemma~\ref{existence of h},
there exists a solution $\tilde{\phi}_{t}$ of the equation \eqref{eq-t}
for any $t<1$.
Take a sequence $t_{i}\nearrow 1$ and define $\tilde{\phi}_{i}$ as the
solution of
\eqref{eq-t} at $t=t_{i}$.

If $\sup_{i}d_{1,G}(\psi_{0},\tilde{\phi}_{i})<+\infty$,
then Proposition~\ref{G-orbit} stated below and Lemma~\ref{Kbounded}
show
the existence of an $h$-modified extremal K\"{a}hler metric.

We consider the case where
$\sup_{i}d_{1,G}(\psi_{0},\tilde{\phi}_{i})=+\infty$ 
to show contradiction.
We follow the same argument as
the proof of Proposition~\ref{bounded of K}.
Take $g_{i}\in G$ such that $\phi_{i}:=g_{t}[\tilde{\phi}_{i}]$
satisfies
$d_{1,G}(\psi_{0},\tilde{\phi}_{i})
\leqq{d_{1}(\psi_{0},\phi_{i})}\leqq
d_{1,G}(\psi_{0},\tilde{\phi}_{i})+1$.
Since an argument in the proof of Lemma~\ref{existence of h} shows
the functional $K_{V}^{h}$ is bounded along the continuity path
$\eqref{eq-t}$,
$\sup_{i} K_{V}^{h}(\tilde{\phi}_{i})<+\infty$.
By Lemma~\ref{Kbounded}, $\sup_{i}K_{V}^{h}(\phi_{i})<+\infty$.
By the same argument as the proof of Proposition~\ref{bounded of K},
the unit speed geodesic segment $\rho_{i}$ connecting $\psi_{0}$ and
$\phi_{i}$ yields
the limiting geodesic ray $\rho_{\infty}$ starting at $\psi_{0}$.
The geodesic ray $\rho_{\infty}$ decreases $K_{V}^{h}$
since each geodesic ray $\rho_{i}$ decreases $\tilde{K}_{t}$
in Lemma~\ref{existence of h}.
Since $(M,\Omega,V)$ is geodesic stable,
$\rho_{\infty}$ is parallel to a geodesic ray
coming from a holomorphic vector field.
Thus $\rho_{\infty}$ is $d_{1,G}$-bounded.
On the other hand, an argument in the proof
of Proposition~\ref{bounded of K} shows
$\rho_{\infty}$ is not $d_{1,G}$-bounded.
Therefore the situation where
$\sup_{i}d_{1,G}(\psi_{0},\tilde{\phi}_{i})=+\infty$ does not occur.
This completes the proof.
\end{proof}

\begin{proposition}\label{G-orbit}
Suppose the $h$-modified Mabuchi functional $K_{V}^{h}$ is bounded from
below.
Take a sequence $t_{i}\nearrow 1$ and let $\tilde{\phi}_{i}$ be a
solution
the continuity path \eqref{eq-t} at $t=t_{i}$ satisfying
$\tilde{\phi_{i}}\in\mathcal{H}_{V,0}$.
Suppose also $\sup_{i}d_{1,G}(0,\tilde{\phi_{i}})<+\infty$.
Let $\phi_{i}\in\mathcal{H}_{V,0}$ be in the $G$-orbit of
$\tilde{\phi_{i}}$ satisfying
$\sup_{i}d_{1}(0,\phi_{i})<+\infty$.
Then a subsequence of $\{\phi_{i}\}_{i}$ converges to a smooth
$h$-modified extremal
K\"{a}hler metric in $C^{1,\alpha}$-topology for any $0<\alpha<1$.
\end{proposition}

\begin{proof}
We follow an argument in \cite[pp.\ 973--978]{CC2}.
Recall that $\tilde{\phi}_{i}$ is a minimizer of the functional
$\tilde{K}_{t_{i}}:=t_{i}K_{V}^{h}+(1-t_{i})\mathbb{J}$.
When $t_{i}\to 1$, we observe
\begin{equation}\label{infK}
\tilde{K}_{t_{i}}(\tilde{\phi}_{i})=
\inf_{\mathcal{H}_{V}}\tilde{K}_{t_{i}}
\to\inf_{\mathcal{H}_{V}}K_{V}^{h},
\quad
K_{V}^{h}(\tilde{\phi}_{i})\to\inf_{\mathcal{H}_{V}}K_{V}^{h}
\quad\text{and}\quad
(1-t_{i})\mathbb{J}(\tilde{\phi}_{i})\to{0}
\end{equation}
(cf. \cite[Lemma~4.15]{CC2}).
Indeed for any $\e>0$, we can take $\phi^{\e}\in\mathcal{H}_{V}$
satisfying
$K_{V}^{h}(\phi^{\e})\leqq\inf_{\mathcal{H}_{V}}K_{V}^{h}+\e$ to obtain
$$
\tilde{K}_{t_{i}}(\tilde{\phi}_{i})
\leqq{t_{i}K_{V}^{h}(\phi^{\e})+(1-t_{i})\mathbb{J}(\phi^{\e})}.
$$
This yields
$$
\limsup_{i\to\infty}
\tilde{K}_{t_{i}}(\tilde{\phi}_{i})\leqq{K_{V}^{h}(\phi^{\e})}
\leqq\inf_{\mathcal{H}_{V}}K_{V}^{h}+\e.
$$
Since $\mathbb{J}\geqq{0}$, together with the above inequality, we have
$$
t_{i}\inf_{\mathcal{H}_{V}}K_{V}^{h}
\leqq{t_{i}K_{V}^{h}(\tilde{\phi}_{i})}
\leqq{K_{t_{i}}(\tilde{\phi}_{i})}
\leqq\inf_{\mathcal{H}_{V}}K_{V}^{h}+\e
$$
which shows \eqref{infK}.

We next observe
\begin{equation}\label{Hphi}
\sup_{i}H(\phi_{i})=
\sup_{i}\int_{M}\log\frac{\omega_{\phi_{i}}^{n}}{\omega_{0}^{n}}
\omega_{\phi_{i}}^{n}<+\infty
\end{equation}
(cf. \cite[Lemma~4.18]{CC2}).
Indeed by Lemma~\ref{Kbounded},
$\sup_{i}K_{V}^{h}(\tilde{\phi}_{i})=\sup_{i}K_{V}^{h}(\phi_{i})$
which is finite by the inequality \eqref{infK}.
Since the functionals $\mathbb{J}_{-\mathrm{Ric}}$ and
$\mathbb{J}_{V}^{h}$ in the definition
of $K_{V}^{h}$ are controlled by the $d_{1}$-distance, then the estimate
\eqref{Hphi} follows.

Let $g_{i}\in\mathrm{Aut}_{0}(M,V)$ be the automorphism such that
$\phi_{i}=g_{i}[\tilde{\phi}_{i}]$.
Define $\omega_{i}=g_{i}^{*}\omega_{0}=\omega_{0}+\deldel h_{i}$ with
$\sup_{M}h_{i}=0$.
Also define
$\theta_{i}=g_{i}^{*}\theta_{V}^{(\tilde{\phi}_{i})}
=\theta_{V}^{(\phi_{i})}$.
A direct computation shows that $\phi_{i}$ satisfies
\begin{equation}\label{eq-phi}
\omega_{\phi_{i}}^{n}=e^{F_{i}}\omega_{0}^{n}
\quad\text{and}\quad
\square_{\omega_{\phi_{i}}}F_{i}=
\Big(\underline{S}-\frac{1-t_{i}}{t_{i}}n+h(\theta_{i})\Big)+
\Lambda_{\omega_{\phi_{i}}}\Big(\mathrm{Ric}(\omega_{0})
-\frac{1-t_{i}}{t_{i}}\omega_{i}\Big)
\end{equation}
(cf. \cite[Lemma~4.14]{CC2}).

We are now in position to prove that
a subsequence of $\{\phi_{i}\}_{i}$ converges to a smooth $h$-modified
extremal K\"{a}hler 
metric in $C^{1,\alpha}$-topology for any $0<\alpha<1$.
Our equation \eqref{eq-t} is different from Chen-Cheng's one by the term
$h(\theta_{i})$.
Since $\|h(\theta_{i})\|_{L^{\infty}}$ is uniformly bounded,
we can apply estimates established by Chen-Cheng \cite{CC2}.
Define $$f_{i}=\frac{1-t_{i}}{t_{i}}h_{i}.$$
By \cite[Lemma~4.20]{CC2}, there exists a constant $C=C(\omega_{0})>0$
such that
for any $p>1$, there exists $\e_{p}>0$ satisfying 
$\| e^{-f_{i}} \|_{L^{p}(\omega_{0}^{n})}\leqq{C}$ when
$t_{i}\in(1-\e_{p},1)$
By \cite[Proposition~4.21]{CC2} (see also \cite[Theorem~3.3]{CC2}),
there exists a constant $C=C(p,\omega_{0}, \sup_{i}H(\phi_{i}))>0$ such
that
$$
\|F_{i}+f_{i}\|_{W^{1,2p}}+\|n+
\square_{\omega_{0}}\phi_{i}\|_{L^{p}(\omega_{0}^{n})}\leqq{C}.
$$
By taking a subsequence, we get $\mathrm{Im}V$-invariant functions
$\phi_{*}\in W^{2,p}$ and $F_{*}\in W^{1,p}$ for any $1<p<+\infty$
such that
\begin{itemize}
\item $\phi_{i}\to\phi_{*}$ in $C^{1,\alpha}$ for any $0<\alpha<1$
and $\deldel\phi_{i}\to\deldel\phi_{*}$ weakly in $L^{p}$,
\item $F_{i}+f_{i}\to F_{*}$ in $C^{\alpha}$ for any $0<\alpha<1$
and $\nabla(F_{i}+f_{i})\to\nabla F_{*}$ weakly in $L^{p}$,
\item $\omega_{\phi_{i}}^{k}\to\omega_{\phi_{*}}^{k}$ 
weakly in $L^{p}$ for any $1\leqq{k}\leqq{n}$.
\end{itemize}
An argument in \cite[Proof of Proposition~4.23]{CC2} shows that
$\phi_{*}$ is
a weak solution in the following sense:
$$
\omega_{\phi_{*}}^{n}=e^{F_{*}}\omega_{0}^{n},
$$
and for any $u\in{C^{\infty}(M;\mathbb{R})}$,
$$
-\int_{M}d^{c}
F_{*}\wedge{du}\wedge\frac{\omega_{\phi_{*}}^{n-1}}{(n-1)!}
=-\int_{M}u\Big(\underline{S}+h(\theta_{V}^{(\phi_{*})})\Big)
\frac{\omega_{\phi_{*}}^{n}}{n!}
+\int_{M}u\mathrm{Ric}(\omega_{0})\wedge\frac{\omega_{\phi_{*}}^{n}}{n!}
$$
where $d^{c}:=\frac{\sqrt{-1}}{2}(\partial-\overline{\partial})$.
By \cite[Theorem~1.1]{CH12}, there exists a constant $C>1$ such that
$C^{-1}\omega_{0}\leqq\omega_{\phi_{*}}\leqq{C\omega_{0}}$ and
$\omega_{\phi_{*}}\in W^{3,p}$.
Therefore the standard elliptic regularity argument shows $\phi_{*}$ is
a smooth $h$-modified
extremal K\"{a}hler metric. This completes the proof.
\end{proof}

\section{Existence for a $\sigma$-soliton
and a $\sigma$-extremal K\"{a}hler metrics}\label{sigma-soliton}

In this section we work on a Fano manifold $M$ with
$\Omega=2\pi{c_{1}(M)}$
to prove Theorem~\ref{2nd thm}. Let $G:=\mathrm{Aut}_{0}(M,V)$.

\begin{theorem}\label{3rd thm}
There exists a $\sigma$-soliton in $2\pi c_{1}(M)$ if and only if
the $\sigma$-Ding functional is 
$G$-invariant and coercive.
Moreover the condition of $G$-invariance can be removed.
\end{theorem}

\begin{proof}
We check that the data 
$\mathcal{R}=\mathcal{H}_{V,0}$, $d=d_{1}$, $F=D_{V}^{h}$ with
$h(s)=1-e^{-\sigma(s)}$ and
$G=\mathrm{Aut_{0}}(M,V)$
enjoy the properties (P1)--(P6) in \cite[Section 4.2]{Dar19} to apply
\cite[Theorem~4.7]{Dar19}.

\begin{description}
\item[(P1)] This is due to \cite{Bern15} and Lemma~\ref{Jaffine}.
\item[(P2)] This can be proved by applying the monotonicity of
$E_{V}^{h}$ with $h(s)=1-e^{-\sigma(s)}$ 
stated in Lemma~\ref{Jaffine} and
\cite[Lemmas~5.15, 5.20 and 5.29]{DR17}
where we replace $\mathrm{AM}_{X}$ and $F^{X}$ in \cite{DR17} by 
$E_{V}^{h}$ and $D_{V}^{h}$ with $h(s)=1-e^{-\sigma(s)}$ respectively.
\item[(P3)] : This is due to \cite[Theorem~1.1]{SaTa} for regularity
argument by an application of a geometric flow
which is based on \cite{SoTi17, SzTo11}.
\item[(P4)]  This is due to \cite[Lemma~5.9]{DR17}.
\item[(P5)]  This is due to (P3) and \cite[Theorem~C]{Ma03}.
\item[(P6)]  This follows from the cocycle condition for $D_{V}^{h}$.
\end{description}

Finally, 
we follow the same argument as in the proof of Theorem~\ref{main thm'}
to conclude that $D_{V}^{h}$ is $G$-invariant
provided the coercivity of $D_{V}^{h}$.
\end{proof}

Before proving Theorem~\ref{2nd thm}, we note that
a $\sigma$-soliton $\omega_{\sigma\text{-sol}}$ and
a $\sigma$-extremal metric $\omega_{\sigma\text{-ext}}$ are different
in general.
Indeed, if
$\omega_{\sigma\text{-sol}}=\omega_{\sigma\text{-ext}}=:\omega$,
or equivalently $S(\omega)-n=1-e^{\rho_{\omega}}$,
then an argument in \cite[Section 3.2]{Ni23} shows $\omega$ is
a K\"{a}hler-Einstein metric.

\begin{proof}[Proof of Theorem~\ref{2nd thm}], 
By the inequality \eqref{KgeqD},
if the $\sigma$-Ding functional is coercive
then so is the $\sigma$-Mabuchi functional.
Therefore Theorems \ref{main thm} and \ref{3rd thm} show that
the existence of a $\sigma$-soliton implies the existence of a
$\sigma$-extremal K\"{a}hler metric.
\end{proof}

\begin{example}
Let $M$ be a toric Fano manifold and $\Omega=2\pi{c_{1}(M)}$.
Then, by \cite[Lemma~2.1]{WaZh04}, there exists a unique holomorphic
vector field $V_{0}\in\mathfrak{X}_{0}(M)=\mathfrak{X}(M)$ satisfying
$F_{V_{0}}^{h}(X)=0$ for any $X\in\mathfrak{X}(M)$,
where $h(s)=1-e^{s+C}$ and
$C:=\log\left(
\mathrm{Vol}/\int_{M}e^{\theta_{V_{0}}^{(\omega)}}\omega^{n}\right)$
for some $\omega\in\Omega=2\pi{c_{1}(M)}$.
Here, we note that this constant $C$ does not depend on the choice of
$\omega\in\Omega=2\pi{c_{1}(M)}$.
In view of Wang-Zhu's result \cite[Theorem~1.1]{WaZh04},
$M$ always admits a K\"{a}hler-Ricci soliton, i.e.,
there exists $\omega_{0}\in\Omega=2\pi{c_{1}(M)}$ satisfying
$\mathrm{Ric}(\omega_{0})-\omega_{0}=L_{V_{0}}\omega_{0}$.
Therefore, by Theorem~\ref{2nd thm}, we have a
$\sigma$-extremal metric $\omega_{1}\in\Omega=2\pi{c_{1}(M)}$ for
$\sigma(s)=-s-C$,
i.e., $\omega_{1}$ satisfies
$$
S(\omega_{1})-n=1-
\frac{\mathrm{Vol}}
{\int_{M}e^{\theta_{V_{0}}^{(\omega_{1})}}\omega_{1}^{n}}
e^{\theta_{V_{0}}^{(\omega_{1})}}.
$$
\end{example}




\begin{thebibliography}{99}

\bibitem{AJL23}
\textsc{V.~Apostolov, S.~Jubert and A.~Lahdili},
Weighted K-stability and coercivity with applications to extremal
K\"{a}hler and Sasaki metrics,
Geom. Topol.27(2023), no.8, 3229--3302.


\bibitem{BB17}
\textsc{R.~Berman and B.~Berndtsson},
Convexity of the K-energy on the space of K\"{a}hler metrics 
and uniqueness of extremal metrics,
J.\ Amer.\ Math.\ Soc.\ 30 (2017), no.\ 4, 1165--1196.


\bibitem{BBEGZ19}
\textsc{R.~Berman, S.~Boucksom, P.~Eyssidieux, V.~Guedj and
A.~Zeriahi},
K\"{a}hler-Einstein metrics and the K\"{a}hler-Ricci flow on log Fano
varieties,
J.\ Reine Angew.\ Math.\ 751 (2019), 27--89.

\bibitem{BDL20}
\textsc{R.~Berman, T.~Darvas and C.~H.~Lu},
Regularity of weak minimizers of the K-energy 
and applications to properness and K-stability,
Ann.\ Sci.\ \'{E}c.\ Norm.\ Sup\'{e}r.\ (4) 53 (2020), no\. 2,
267--289.

\bibitem{BW14}
\textsc{R.~Berman, D.~Witt~Nystr\"{o}m}, 
Complex optimal transport and the pluripotential theory of
K\"{a}hler-Ricci solitons,
arXiv:1401.8264.

\bibitem{Bern15}
\textsc{B.~Berndtsson}, 
A Brunn-Minkowski type inequality for Fano manifolds and some uniqueness
theorems in K\"{a}hler geometry, 
Invent.\ Math.\ 200:1 (2015), 149--200.


\bibitem{Ca82}
\textsc{E.~Calabi},
Extremal K\"{a}hler metrics, Seminar on Differential Geometry
(1982), Vol.102, Princeton University Press.

\bibitem{C18}
\textsc{X.~Chen}, 
On the existence of constant scalar curvature K\"{a}hler metric:
A new perspective,
Ann.\ Math.\ Qu\'{e}. 42 (2018), no.\ 2, 169--189.

\bibitem{C20}
\textsc{X.~Chen}, 
The space of K\"{a}hler metrics,
J.\ Differential Geom.\ 56 (2000), no.\ 2, 189--234.

\bibitem{CC1}
\textsc{X.~Chen and J.~Cheng},
On the constant scalar curvature K\"{a}hler
metrics (I)--A priori estimates,
J.\ Amer.\ Math.\ Soc.\ 34 (2021), no.\ 4, 909--936.

\bibitem{CC2}
\textsc{X.~Chen and J.~Cheng},
On the constant scalar curvature K\"{a}hler
metrics (II)--Existence results,
J. Amer. Math. Soc. 34 (2021), no. 4, 937--1009.





\bibitem{CH12}
\textsc{X.~Chen and W.~He}, 
The complex Monge-Amp\`{e}re equation on compact K\"{a}hler manifolds,
Math.\ Ann.\ 354 (2012), no.\ 4, 1583--1600.





\bibitem{Dar15}
\textsc{T.~Darvas},
The Mabuchi geometry of finite energy classes,
Adv.\ Math.\ 285 (2015) 182--219.

\bibitem{Dar17}
\textsc{T.~Darvas},
The Mabuchi completion of the space of K\"{a}hler potentials,
Amer.\ J.\ Math.\ 139 (2017), no.\ 5, 1275--1313

\bibitem{Dar19}
\textsc{T.~Darvas},
Geometric pluripotential theory on K\"{a}hler manifolds,
In Advances in complex geometry, volume 735 of
Contemp.\ Math.\, pages 1--104. Amer.\ Math.\ Soc.\,
Providence, RI, 2019. 

\bibitem{DR17}
\textsc{T.~Darvas and Y.~A.Rubinstein},
Tian's properness conjectures and Finsler geometry of the space of
K\"{a}hler metrics,
J.\ Amer.\ Math.\ Soc.\ 30 (2017), no.\ 2, 347--387.







\bibitem{Fu83}
\textsc{A.~Futaki},
On compact K\"{a}hler manifolds of constant scalar curvature,
Proc.\ Japan Acad.\ 59 (1983), pp.\ 401--402.

\bibitem{Fubook}
\textsc{A.~Futaki}, 
K\"{a}hler-Einstein metrics and integral invariants, 
Lecture Notes in Math.\, 1314, Springer-Verlag, Berlin-Heidelberg-New
York, 1988.


\bibitem{Fu02}
\textsc{A.~Futaki},
Some invariant and equivariant cohomology classes of the space of
K\"{a}hler metrics,
Proc.\ Japan Acad.\ Ser.\ A Math.\ Sci.\ 78 (2002), no.\ 3, 27--29. 

\bibitem{FM}
\textsc{A.~Futaki and T.~Mabuchi}, 
Bilinear forms and extremal K\"{a}hler vector fields associated with
K\"{a}hler classes,
Math.\ Ann.\ 301(1995), 199--210.




\bibitem{HaLi20}
\textsc{J.~Han and C.~Li},
On the Yau-Tian-Donaldson conjecture for generalized K\"{a}hler-Ricci
soliton equations,
Comm.\ Pure Appl.\ Math.\ 76 (2023), no.\ 9, 1793--1867.

\bibitem{Ha19}
\textsc{Y.~Hashimoto},
Existence of twisted constant scalar curvature K\"{a}hler metrics with a
large twist,
Math.\ Z.\ 292 (2019), 791--803.

\bibitem{He19}
\textsc{W.~He},
On Calabi's extremal metric and properness,
Trans.\ Amer.\ Math.\ Soc.\ 372 (2019), no.\ 8, 5595--5619.

\bibitem{He19pre}
\textsc{W.~He},
On Calabi's extremal metric and properness,
arXiv:1801.07636v1.

\bibitem{Hi19-2}
\textsc{T.~Hisamoto},
Mabuchi's soliton metric and relative D-stability,
Amer.\ J.\ Math.\ 145 (2023), no.\ 3, 765--806.



\bibitem{Kobayashi72a}
\textsc{S.~Kobayashi},
Transformation groups in differential geometry,
Springer-Verlag, Heidelberg, 1972.


\bibitem{Koi90}
\textsc{N.~Koiso},
On rotationally symmetric Hamilton's equation for K\"{a}hler-Einstein
metrics,
in ``Recent Topics in differential and analytic geometry'',
Adv.\ Stud.\ in Pure Math.\ 18-I, Kinokuniya and Academic Press,
Tokyo and Boston, 1990, pp.\ 327--337.




\bibitem{Lah19}
\textsc{A.~Lahdili},
K\"{a}hler metrics with constant weighted scalar curvature and weighted
K-stability,
Proc. Lond. Math. Soc. (3)119(2019), no.4, 1065--1114.

\bibitem{Lah23}
\textsc{A.~Lahdili},
Convexity of the weighted Mabuchi functional and the uniqueness of
weighted extremal metrics,
Math. Res. Lett.30(2023), no.2, 541--576.



\bibitem{LZ}
\textsc{Y.~Li and B.~Zhou},
Mabuchi metrics and properness of the modified Ding functional,
Pacific J.\ Math.\ 302 (2019), no.\ 2, 659--692.


\bibitem{Ma01}
\textsc{T.~Mabuchi},
K\"{a}hler-Einstein metrics for manifolds with nonvanishing Futaki
character,
Tohoku Math.\ J.\, (2) 53 (2001), 171--182.

\bibitem{Ma03}
\textsc{T.~Mabuchi},
Multiplier Hermitian structures on K\"{a}hler manifolds, 
Nagoya Math. J. 170 (2003), 73--115.

\bibitem{Mabook}
\textsc{T.~Mabuchi},
Test configurations, stabilities and canonical K\"{a}hler metrics--complex
geometry by the energy method, 
SpringerBriefs in Mathematics, Springer, Singapore, [2021].

 






\bibitem{NN22}
\textsc{Y.~Nakagawa and S.~Nakamura},
 Multiplier Hermitian-Einstein metrics on Fano manifolds of KSM-type, 
 Tohoku Math.\ J.\ (1) 76 (2024), no.\ 1. 127--152.

\bibitem{Na19}
\textsc{S.~Nakamura},
Generalized K\"{a}hler Einstein metrics and uniform stability for toric
Fano manifolds,
Tohoku Math.\ J.\ (2) 71 (2019), no.\ 4, 525--532.


\bibitem{Ni23}
\textsc{Y.~Nitta},
On relations between Mabuchi's generalized K\"{a}hler-Einstein metrics
and various canonical K\"{a}hler metrics,
Proc.\ Amer.\ Math.\ Soc.\ 151 (2023), no.\ 7, 3083--3087.








\bibitem{SaTa}
\textsc{S.~Saito and R.~Takahashi},
Regularity of weak multiplier Hermitian metrics and applications,
preprint.


\bibitem{SoTi17}
\textsc{J.~Song and G.~Tian},
The K\"{a}hler-Ricci flow through singularities,
Invent.\ Math.\ 207 (2017), no.\ 2, 519--595.


\bibitem{Szbook}
\textsc{G.~Sz\'eKelyhidi},
An introduction to extremal K\"{a}hler metrics,
Grad.\ Stud.\ Math.\, 152,
American Mathematical Society, Providence, RI, 2014.

\bibitem{SzTo11}
\textsc{G.~Sz\'{e}kelyhidi and V.~Tosatti},
Regularity of weak solutions of a complex Monge-Amp\`{e}re equation,
Anal.\ PDE 4 (2011), no.\ 3, 369--378.




\bibitem{Tianbook}
\textsc{G.~Tian},
Canonical metrics in K\"{a}hler geometry,
Lectures in Mathematics ETH Z\"{u}rich,
Birkh\"{a}user Verlag, Basel, 2000. 

\bibitem{TZ02}
\textsc{G.~Tian and X.~Zhu},
A new holomorphic invariant and uniqueness of K\"{a}hler-Ricci solitons,
Comment.\ Math.\ Helv.\ {\bf 77} (2002), pp.\ 297--325.


\bibitem{WaZh04}
\textsc{X.~Wang and X.~Zhu},
K\"{a}hler-Ricci solitons on toric manifolds with positive first Chern
class,
Adv.\ Math.\ 188 (1) (2004) 87--103.

\bibitem{Ya21}
\textsc{Y.~Yao},
Mabuchi Solitons and Relative Ding Stability of Toric Fano Varieties,
Int.\ Math.\ Res.\ Not.\ IMRN (2022),
no.\ 24, 19790--19853.

\bibitem{Ya22}
\textsc{Y.~Yao}, 
Relative Ding Stability and an Obstruction to the Existence of Mabuchi
Solitons,
J.\ Geom.\ Anal.\ 32, 105 (2022).

\bibitem{Zh00}
\textsc{X.~Zhu}
K\"{a}hler-Ricci soliton typed equations on compact complex manifolds
with $C_{1}(M)>0$,
J.\ Geom.\ Anal.\ 10 (2000), no.\ 4, 759--774.

\end{thebibliography}
\end{document}